\documentclass{article}
\pdfoutput=1
\usepackage{amsrefs}
\usepackage{amssymb}
\usepackage{amsmath}
\usepackage{amsthm}
\usepackage{graphicx}

\DeclareMathOperator{\im}{im}
\DeclareMathOperator{\Sing}{Sing}
\DeclareMathOperator{\Aut}{Aut}
\newcommand{\PP}{{\mathbb{P}}}

\theoremstyle{definition}
\newtheorem{definition}{Definition}[section]
\newtheorem{example}[definition]{Example}
\newtheorem{notation}[definition]{Notation}
\newtheorem{remark}[definition]{Remark}
\theoremstyle{plain}
\newtheorem{lemma}[definition]{Lemma}
\newtheorem{proposition}[definition]{Proposition}
\newtheorem{theorem}[definition]{Theorem}
\newtheorem{corollary}[definition]{Corollary}

\begin{document}
\title{\textbf{Star points on smooth hypersurfaces}}
\author{Filip Cools\footnote{K.U.Leuven, Department of Mathematics,
Celestijnenlaan 200B, B-3001 Leuven, Belgium, email:
Filip.Cools@wis.kuleuven.be; the author is a Postdoctoral Fellow of the Research Foundation - Flanders
(FWO).} \hspace{0.2mm} and \addtocounter{footnote}{5}Marc Coppens\footnote{Katholieke
Hogeschool Kempen, Departement Industrieel Ingenieur en Biotechniek,
Kleinhoefstraat 4, B-2440 Geel, Belgium; K.U.Leuven, Department of Mathematics,
Celestijnenlaan 200B, B-3001 Leuven, Belgium; email: Marc.Coppens@khk.be}}

\maketitle {\footnotesize \emph{\textbf{Abstract.---} A point $P$ on a smooth hypersurface $X$ of degree $d$ in $\mathbb{P}^N$ is called a star point if and only if the intersection of $X$ with the embedded tangent space $T_P(X)$ is a cone with vertex $P$. This notion is a generalization of total inflection points on plane curves and Eckardt points on smooth cubic surfaces in $\mathbb{P}^3$. We generalize results on the configuration space of total inflection points on plane curves to star points. We give a detailed description of the configuration space for hypersurfaces with two or three star points. We investigate collinear star points and we prove that the number of star points on a smooth hypersurface is finite. \\ \\
\indent \textbf{MSC.---} 14J70, 14N15, 14N20}}
\\ ${}$

\section{Introduction}

In our papers \cite{CC1} and \cite{CC2}, we studied the locus of
smooth plane curves of degree $d$ containing a given number of total
inflection points. As a starting point in \cite{CC1}, we examined
possible configurations of points and lines that can appear as such
total inflection points and the associated tangent lines. In this paper, we generalize this point of view to higher dimensional
varieties.

In case $C$ is a smooth plane curve and $P$ is a point on $C$, the
intersection $T_P(C)\cap C$ of the tangent line with $C$ is a
$0$-dimensional scheme of length $d$. The point $P$ is called a total inflection
point of $C$ if and only if this scheme has maximal multiplicity (being $d$) at $P$. Therefore, as a generalization, we
consider points $P$ on smooth hypersurfaces $X$ of degree $d$ in
$\mathbb{P}^N$ such that the intersection $T_P(X)\cap X$ of the
tangent space with $X$ has multiplicity $d$ at $P$. This condition
is equivalent to $T_P(X)\cap X$ being a cone with vertex $P$ in
$T_P(X)$. Because of this pictorial description, we call such point $P$ a star point on $X$.

If $P$ is a total inflection point of a plane curve $C$ and $L=T_P(C)$, the intersection scheme $C\cap L$ is fixed: it is the divisor $dP$ on $L$. Therefore, a configuration space for total inflection points on plane curves is defined using pairs $(L,P)$ with $L$ a line on $\mathbb{P}^2$ and $P$ a point on $L$. In case $N>2$, $P$ is a star point on a smooth hypersurface $X$ and $\Pi=T_P(X)$, then $X\cap\Pi$ is not fixed in advance. Therefore, a configuration space for star points on hypersurfaces in $\mathbb{P}^N$ of degree $d$ is defined using triples $(\Pi,P,C)$ with $\Pi$ a hyperplane in $\mathbb{P}^N$, $P$ a point on $\Pi$ and $C$ a cone hypersurface of degree $d$ in $\Pi$ with vertex $P$. When restricting to smooth hypersurfaces (as done in this paper), one can assume $C$ being smooth outside $P$.

As in the case of plane curves, we prove that there is a strong relation between the space of hypersurfaces having a given number $e$ of star points and the space of associated configurations. Using this relation we find a lower bound for the dimension of the components of such space of hypersurfaces. We call this lower bound the expected dimension. We give a complete description of the configuration space associated to hypersurfaces $X$ having two and three star points. In case $e=2$, we have two components with one having the expected dimension and the other one having larger dimension. The component with unexpected dimension corresponds to the case when the line spanning two star points is contained in $X$. In case $e=3$, we have components of the expected dimension but also components with larger dimension.

In case $X$ is a smooth hypersurface of degree $d$ in $\mathbb{P}^N$ and $L$ is a line on $X$, then there are at most two star points of $X$ on $L$. On the other hand, there do exist smooth hypersurfaces $X$ in $\mathbb{P}^N$ of degree $d$ such that for some line $L$ the intersection $X\cap L$ consists of $d$ star points of $X$. In case $X$ is a smooth hypersurface of degree $d$ in $\mathbb{P}^N$ and $L\not\subset X$ is a line containing at least $d-1$ star points of $X$, then $L$ contains exactly $d$ star points. This generalizes the classical similar fact on inflection points on cubic curves and it is an extra indication that the concept of star points is the correct generalization of total inflection points on plane curves. We show that the case with three star points and the case with collinear star points are the basic cases.

Although there exist smooth hypersurfaces having many star points (e.g. the Fermat hypersurfaces), we prove that there are only finitely many if $d\geq 3$.  \\

Star points are intensively studied in the case of smooth cubic surfaces $X$ in $\mathbb{P}^3$. In this case, a star point $P$ is
a point such that $3$ lines on $X$ meet at $P$. The study of the lines
on a smooth cubic surface is very classical; it is
well-known that there are exactly $27$ lines on such a surface (see e.g. \cite{Har}).
The study of smooth cubic surfaces having $3$ lines through one point started with a paper by F.E.
Eckardt (see \cite{Eck}); therefore such points are often called
Eckardt points. A classification of smooth cubic surfaces according
to their star points (also considering the real case) was
done in \cite{Seg}, using so-called harmonic homologies. In his Ph-D
thesis \cite{Ngu1}, Nguyen gave a classification for the complex
case using the description of a smooth cubic surface as a blowing-up
of $6$ points on $\mathbb{P}^2$ (see also \cite{Ngu2}). From that point of view, the author
also studied singular cubic surfaces (see also \cites{Ngu3,Ngu4}).

Eckardt points on cubic surfaces also appear in the recent paper
\cite{DVG}. For most of the smooth cubic surfaces $S$, there
exists a $K3$-surface $X$ being the minimal smooth model of the
quartic surface $Y$ defined by the Hessian associated to the
equation of $S$. In general, those $K3$-surfaces have Picard number 16
and one has an easy description for the generators of the Neron-Severi
group $NS(X)$. Cubic surfaces $S$ with Eckardt points give rise to
higher Picard numbers. In particular, one also finds $K3$-surfaces with
Picard number 20 (the so-called singular $K3$-surfaces). In those
cases, the Eckardt points give rise to easily determined "new" curves
on $X$.

The generalization to higher dimensional hypersurfaces also occurs
in \cite{CP}. In this paper, the authors study the log
canonical threshold of hyperplane sections for a smooth hypersurface
$X$ of degree $N$ in $\mathbb{P}^N$. Note that such hypersurfaces are
embedded by means of the anticanonical linear system. Assuming the
log minimal model program, they obtain a strong relation between
hyperplane sections coming from tangent spaces at star points of $X$
and the minimal value for the log canonical threshold. Instead of
talking about star points, the authors call them generalized Eckardt
points.

A star point on a smooth hypersurface $X$ in $\mathbb{P}^N$ of degree $d$ gives rise to a subset of the Fano scheme $\mathbb{F}(X)$ of lines on $X$ of dimension $N-3$. The Fano scheme of hypersurfaces is intensively studied (see \cites{Beh,HMP,LaTo}). For a general hypersurface, the Fano scheme is smooth of dimension $2N-d-3$. This also holds for non-general hypersurfaces if $N$ is large with respect to $d$. On the other hand, in case $d>N$, the dimension of the Fano scheme of non-general hypersurfaces (e.g. Fermat hypersurfaces) might be larger than $2N-d-3$. However, the finiteness of the number of star points for smooth hypersurfaces $X$ in $\mathbb{P}^N$ of degree $d\geq 3$ implies that there exists no $(N-2)$-dimensional subset $B$ of $\mathbb{F}(X)$ such that each line $L$ of $B$ meets a fixed curve $\Gamma\subset X$.

Star points are also related to so-called Galois points. Let $X$ be
a smooth hypersurface in $\mathbb{P}^N$, let $P$ be a point in
$\mathbb{P}^N$ and let $H$ be a hyperplane not containing $P$. The
point $P$ is called a Galois point of $X$ in case projection of $X$
to $H$ with center $P$ corresponds to a Galois extension of function
fields. In case $P$ is an inner Galois point on $X$ ("inner" means
$P$ is a point on $X$) then $P$ is a star point of $X$. This is
proved for quartic surfaces in \cite{Yos1}*{Corollary 2.2} and in
general in \cite{Yos2}*{Corollary 6}. \\

The article is structured as follows. In Section \ref{section def}, we give the definition of a star point on a hypersurface and prove a general result on configurations of star points (Theorem \ref{theorem main}). We use a connection between star points and polar hypersurfaces (Lemma \ref{lem3}) to determine all star points on a Fermat hypersurface (Example \ref{exa2}). In Section \ref{section collinear star points}, we study star points on a line $L$ in case $L\subset X$ (Proposition \ref{prop2}) and in case $L\not\subset X$ (Proposition \ref{prop3} and Theorem \ref{theo1}). In Section \ref{section number}, we prove that the number of star points on a smooth hypersurface is always finite (Theorem \ref{theo finite}). In Section \ref{section conf}, we prove the lower bound for the dimension of the configuration space associated to a given number $e$ of star points (Corollary \ref{cor}). We also show that the case of collinear star points and three star points are basic (Proposition \ref{prop5}). In Section \ref{section two starpoints} and Section \ref{section three starpoints}, we investigate the number of components of the configuration set of hypersurfaces with two or three star points.

\section{Definition and first results} \label{section def}

We work over the field $\mathbb{C}$ of complex numbers.

\begin{notation}\label{not1}
Let $X$ be a hypersurface in $\mathbb{P}^N$ and let $P$ be a smooth
point on $X$. We write $T_P(X)$ to denote the tangent space of $X$
inside $\mathbb{P}^N$.
\end{notation}

\begin{definition}\label{def1}
Let $X$ be a hypersurface of degree $d$ in $\mathbb{P}^N$ ($N\geq
2$) and let $P$ be a smooth point on $X$. We say that $P$ is a \textit{star
point} on $X$ if and only if the intersection $T_P(X)\cap X$ (as a
scheme) has multiplicity $d$ at $P$.
\end{definition}

\begin{remark}\label{rem1}
Some remarks related to the previous definition.
\begin{itemize}
\item Let $X$ and $P$ be as in Definition \ref{def1} and let $N\geq 3$. Then $P$ is a
star point of $X$ if and only if the intersection $T_P(X)\cap X$ is
a cone of degree $d$ with vertex $P$ inside $T_P(X)$. In particular,
$X$ has to be irreducible.
\item Let $C$ be a plane curve of degree $d$ and let $P$ be a smooth
point on $C$. Then $P$ is a star point of $C$ if and only if $P$ is
a total inflection point for $C$ (i.e. the tangent line to $C$ at
$P$ intersects $C$ with multiplicity $d$ at $P$). Hence the concept
of a star point is a generalization of the concept of a total
inflection point of a plane curve.
\item Let $X$ be a surface of degree $d$ in projective space
$\mathbb{P}^3$ and let $P$ be a star point on $X$. Then the
intersection of the tangent plane to $X$ at $P$ with $X$ is (as a
set) a union of lines through $P$ in that tangent plane. In case
$d=3$ and $X$ is smooth, the intersection is the union of three lines on $X$ through $P$ (here we also use Lemma \ref{lem1}).
Classically, such a point $P$ is called an Eckardt point on $X$.
\end{itemize}
\end{remark}

\begin{lemma}\label{lem1}
Let $X$ be a smooth hypersurface of degree $d$ in $\mathbb{P}^N$
($N\geq 3$) and let $P$ be a star point on $X$. Then the
intersection $X\cap T_P(X)$ is smooth outside the vertex $P$.
\end{lemma}

\begin{proof}
Let $C=T_P(X)\cap X$ and assume $Q$ is a singular point of $C$
different from $P$. Since $C$ is a cone with vertex $P$ it follows
that all points on the line $<P,Q>$ are singular points on $C$.

Choose coordinates $(X_0:\ldots:X_N)$ in $\mathbb{P}^N$
such that $P=(1:0: \ldots :0)$, the hypersurface $T_P(X)$ has
equation $X_N=0$ and $Q=(1:1:0: \ldots :0)$.
The equation of $X$ is of the form
\begin{displaymath}
F(X_0,\ldots,X_N)=X_N G(X_0,\ldots,X_N)+H(X_1,\ldots,X_{N-1}),
\end{displaymath}
with G (resp. H) homogeneous of degree $d-1$ (resp. $d$),
$G(1,0,\ldots,0)\neq 0$ (because $Z(F)$ has to be smooth at $P$
with tangent space $Z(X_N)$), $H(1,0,\ldots,0)=0$ (because $Q\in
T_P(X)\cap X$) and $(\partial H/\partial X_i)(1,0,\ldots,0)=0$ for
$1\leq i\leq N-1$ (because $C$ is singular at $Q$). Clearly $(a:b:0:
\ldots:0)\in X$ for all $(a:b)\in \mathbb{P}^1$. It follows that
$(\partial F/\partial X_i)(1,c,0,\ldots,0)=0$ for $0\leq i\leq
N-1$ and $(\partial F/\partial X_N)(1,c,0,\ldots,0)=G(1,c,0,\ldots,0)$.
Unless $G(1,T,0,\ldots,0)$ is a constant different
from $0$, it has a zero $c_0$ and then $(1:c_0:0:\ldots:0)$ is a
singular point on $X$, contradicting the assumption that $X$ is
smooth.

In case $G(1,T,0,\ldots,0)$ is a nonzero constant, one has
\begin{displaymath}
G=aX_0^{d-1}+\sum_{i=2}^N X_i G_i(X_0,\ldots,X_N),
\end{displaymath}
with $G_i$ homogeneous of degree $d-2$ and $a\neq 0$. Consider the point
$R=(0:1:0:\ldots :0)\in T_P(X)\cap X$. Still $(\partial F/\partial
X_i)(R)=0$ for $0\leq i\leq N-1$ but also
$(\partial F/\partial X_N)(R)=G(R)=0$, since $\partial F/\partial
X_N=G+X_N\partial G/\partial X_N$. This implies $X$ is singular
at $R$, contradicting the assumptions.
\end{proof}

In this paper we study star points on smooth hypersurfaces. Motivated
by the previous lemma, we introduce the following definition.

\begin{definition}\label{def2}
Let $\mathbb{P}$ be some projective space and let $P$ be a point in
$\mathbb{P}$. A hypersurface $C$ of degree $d$ in $\mathbb{P}$ is
called a \textit{good $P$-cone} of degree $d$ in $\mathbb{P}$ if $C$ is a cone
with vertex $P$ and $C$ is smooth outside $P$.
\end{definition}

Using Definition \ref{def2}, we can restate Lemma \ref{lem1} as
follows. If $X$ is a smooth hypersurface of degree $d$ in
$\mathbb{P}^N$ and if $P$ is a star point on $X$, then $T_P(X)\cap X$
as a scheme is a good $P$-cone of degree $d$ in $T_P(X)$.

We are going to generalize Proposition $1.3$ in \cite{CC1} to the
case of star points. First we introduce some terminology.

\begin{notation}
To a projective space $\mathbb{P}^N$,
we associate the parameter space $\mathcal{P}_d$ of triples $\mathcal{T}=(\Pi,P,C)$ with $\Pi$ a hyperplane in
$\mathbb{P}^N$, $P$ a point on $\Pi$ and $C$ a good $P$-cone of
degree $d$ in $\Pi$. Note that $$\dim (\mathcal{P}_d)=N+(N-1)+{ N+d-2 \choose N-2 }-1=2N+{ N+d-2
\choose N-2 }-2.$$
\end{notation}

Consider such triples $\mathcal{T}_1,
\ldots,\mathcal{T}_e$ with $\mathcal{T}_i=(\Pi_i,P_i,C_i)$ and
$P_i\notin \Pi_j$ for $i\neq j$.

\begin{definition}\label{def3}
We say $\mathcal{L}=\{\mathcal{T}_1,\ldots,\mathcal{T}_e\}$ is
\textit{suited for degree $d$} if there exists a hypersurface $X$ of degree $d$
in $\mathbb{P}^N$ not containing $\Pi_i$ for some $1\leq i\leq e$
and such that $X\cap \Pi_i=C_i$.
\end{definition}

\begin{notation}\label{not2}
Let $\mathcal{L}=\{\mathcal{T}_1,\ldots,\mathcal{T}_e\}$. We write
$$V_d(\mathcal{L})= \{s\in \Gamma(\mathbb{P}^N,
\mathcal{O}_{\mathbb{P}^N}(d))\,|\,C_i\subset Z(s) \},$$ where $Z(s)$ is the zero locus of $s$.
Let $\mathbb{P}_d(\mathcal{L})=\mathbb{P}(V_d(\mathcal{L}))$ be its projectivization, so $\mathbb{P}_d(\mathcal{L})$ is a linear system
of hypersurfaces of degree $d$ in $\mathbb{P}^N$.
For $\mathcal{L}=\emptyset$, we write $V_d=V_d(\mathcal{L})$ and $\mathbb{P}_d=\mathbb{P}_d(\mathcal{L})$ (so $V_d=\Gamma(\mathbb{P}^N,
\mathcal{O}_{\mathbb{P}^N}(d))$).

If $\Pi$ is a hyperplane of $\mathbb{P}^N$, we can consider the restriction map $$r:V_d(\mathcal{L})\rightarrow
\Gamma(\Pi,\mathcal{O}_{\Pi}(d)).$$ Denote the image of $r$ by $V_{\Pi,d}(\mathcal{L})$ and write $\mathbb{P}_{\Pi,d}(\mathcal{L})=\mathbb{P}(V_{\Pi,d}(\mathcal{L}))$ to denote its projectivization. Again, we can consider the case $\mathcal{L}=\emptyset$, and we write $V_{\Pi,d}(\mathcal{L})=V_{\Pi,d}$ and $\mathbb{P}_{\Pi,d}(\mathcal{L})=\mathbb{P}_{\Pi,d}$ in this case.

Note that $\dim(V_d)={ N+d \choose
N }$ and $\dim (V_d(\Pi))={ N-1+d \choose N-1 }$.
\end{notation}

\begin{remark}\label{rem2}
In case $X$ is a quadric in $\mathbb{P}^N$ and $P$ is a smooth point
on $X$, then $T_P(X)\cap X$ is a quadric in $T_P(X)$ singular at $P$.
Since a singular quadric in some projective space is always a cone
with vertex $P$, it follows that $P$ is a star point of $X$.
Therefore all smooth points on a quadric are star points on that
quadric, hence from now on we only consider the case $d\geq 3$.
\end{remark}

\begin{theorem}\label{theorem main}
Assume $\mathcal{L}$ is suited using cones of degree $d\geq 3$.
\begin{enumerate}
\item[(i)] Then $\dim (\mathbb{P}_d(\mathcal{L}))={ d-e+N \choose N }$ (here
${ d-e+N \choose N }=0$ if $e>d$). In case $e\leq d$, a
general element $X$ of $\mathbb{P}_d(\mathcal{L})$ is a smooth
hypersurface of degree $d$.
\item[(ii)] Let $\Pi$ be a hyperplane in $\mathbb{P}^N$ with $P_i\notin
\Pi$ for $1\leq i\leq e$. Let $\Pi'_i=\Pi \cap
\Pi_i$ for $1\leq i\leq e$.

If $e\leq d$, then $\mathbb{P}_{\Pi ,d}(\mathcal{L})$ has dimension
${ d-e+N-1 \choose N-1 }$ and it contains
\begin{displaymath}
\Pi'_1 + \ldots + \Pi'_e + \mathbb{P}_{\Pi,d-e}.
\end{displaymath}
If $e>d$, then $\mathbb{P}_{\Pi ,d}(\mathcal{L})$ has dimension $0$.
\item[(iii)] Let $\Pi$ be a hyperplane as before. Fix a point $P$ in $\Pi$ and a good $P$-cone $C$ of degree $d$ in $\Pi$. Let $\mathcal{T}=(\Pi,P,C)$ and consider $\mathcal{L}'=\mathcal{L}\cap \{\mathcal{T}\}$. Then $\mathcal{L}'$ is suited for degree $d$ if and only if $C\in\mathbb{P}_{\Pi,d}(\mathcal{L})$.
\end{enumerate}
\end{theorem}

\begin{proof}
First we are going to prove that (ii) and (iii) follow from (i). So
assume $\mathcal{L}$ is suited and (i) holds for $\mathcal{L}$. Let
$\Pi$ be a hyperplane in $\mathbb{P}^N$ such that $P_i\notin \Pi$
for $1\leq i\leq e$.

First assume $\dim (\mathbb{P}_d(\mathcal{L}))=0$, i.e.
$\mathbb{P}_d(\mathcal{L})$ contains a unique hypersurface $X$ of
degree $d$ not containing $\Pi_i$ for $1\leq i\leq e$. It follows
that $e>d$ and $\Pi_i\cap X=C_i$ for $1\leq i\leq e$. If $\Pi
\subset X$ then $\Pi'_i$ is a hyperplane inside $\Pi_i$ not
containing $P_i$ and contained in $X$. This contradicts the fact
that $C_i=X\cap \Pi_i$ because $C_i$ is a cone in $\Pi_i$ with
vertex $P_i$. It follows that $\Pi \not\subset X$, hence
$\mathbb{P}_{\Pi ,d}(\mathcal{L})$ is not empty. Clearly it has
dimension zero, thus (ii) holds. Let $P$, $C$, $\mathcal{T}$
$\mathcal{T'}$ be as described in (iii). If $X'\in
\mathbb{P}_d(\mathcal{L}')$, then clearly $X'\in
\mathbb{P}_d(\mathcal{L})$, hence $X'=X$ and we need $X\cap \Pi=C$ (we
already know that $\Pi \subset X$ is impossible), so $C\in
\mathbb{P}_{\Pi ,d}(\mathcal{L})$. Conversely, if $C\in
\mathbb{P}_{\Pi ,d}(\mathcal{L})$, then we have $X\cap \Pi =C$, hence
$X\in \mathbb{P}_d(\mathcal{L}')$. Since we already proved that $\Pi
\not\subset X$ it follows that $\mathcal{L}'$ is suited. This proves
(iii) in this case.

Now assume $e\leq d$, hence $\dim(\mathbb{P}_d(\mathcal{L}))={ d-e+N
\choose N }$. Let $\Pi$ be as in (ii) and let $s\in V_d(\mathcal{L})$
with $r(s)=0$ (where $r$ is as in Notation \ref{not2}), hence $\Pi \subset Z(s)$. It follows that
$\Pi'_i\subset Z(s)$ but also $C_i\subset \Pi_i\cap Z(s)$. Since
$P_i$ is a vertex of the cone $C_i$ and $P_i\notin \Pi'_i$, it
follows that $\Pi_i\subset Z(s)$. This proves that $$\ker(r)=\pi_1.\cdots
.\pi_e.\pi.\Gamma (\mathbb{P}^N,\mathcal{O}_{\mathbb{P}^N}(d-e-1),$$
where we write $\pi$ to denote an equation of
$\Pi$ and so on, hence $\dim (\ker (r))={ d-e+N-1 \choose N }$.
This implies
\begin{equation*}\begin{array}{lll} \dim (\im (r))&=&\dim (V_d(\mathcal{L}))- \dim (\ker (r)) \\
&=& {d-e+N \choose N }+1-{ d-e+N-1 \choose N }={ d-e+N-1 \choose N-1}+1, \end{array} \end{equation*}
hence $\dim (\mathbb{P}_{\Pi ,d}(\mathcal{L}))={ d-e+N-1
\choose N-1 }$. Since $$\pi_1. \cdots .\pi_e.\Gamma
(\mathbb{P}^N,\mathcal{O}_{\mathbb{P}^N}(d-e))\subset
V_d(\mathcal{L}),$$ we have $\Pi'_1 + \ldots +
\Pi'_e+\mathbb{P}_{\Pi,d-e}\subset \mathbb{P}_{\Pi ,d}(\mathcal{L})$.
This finishes the proof of (ii).

Let $P\in \Pi$ with $P_i\notin \Pi$ and $C$ a good $P$-cone of degree
$d$ in $\Pi$. Let $\mathcal{T}=(\Pi,P,C)$ and assume
$\mathcal{L}'=\mathcal{L}\cup \{\mathcal{T} \}$ is suited. There
exists $X'\in \mathbb{P}_d(\mathcal{L}')$ with $X'\cap \Pi =C$.
Since $X'\in \mathbb{P}_d(\mathcal{L})$, it follows $C\in
\mathbb{P}_{\Pi ,d}(\mathcal{L})$. Conversely, assume $C\in
\mathbb{P}_{\Pi ,d}(\mathcal{L})$, hence there exists $X\in
\mathbb{P}_d(\mathcal{L})$ with $X\cap \Pi=C$. Assume $X$ contains a
hyperplane $\Pi_i$ for some $1\leq i\leq e$, then $\Pi \cap
\Pi_i\subset X$. Since $P\notin \Pi_i$ and $X\cap \Pi$ is a cone with
vertex $P$ different from $\Pi$, we obtain a contradiction.
Therefore $\Pi_i\not\subset X$. This proves
$\mathcal{L}'$ is suited, finishing the proof of (iii).

Now we are going to prove (i). First we consider the case $e=1$.
Consider the exact sequence
\begin{displaymath}
0\rightarrow
\Gamma(\mathbb{P}^N,\mathcal{O}_{\mathbb{P}^N}(d-1))\rightarrow
\Gamma(\mathbb{P}^N,\mathcal{O}_{\mathbb{P}^N}(d))\rightarrow
\Gamma(\Pi_1,\mathcal{O}_{\Pi_1}(d))\rightarrow 0.
\end{displaymath}
The cone $C_1\subset \Pi_1$ corresponds to a
subspace of $V_{\Pi_1,d}=\Gamma(\Pi_1,\mathcal{O}_{\Pi_1}(d))$ of dimension one and its
inverse image in $V_d=\Gamma(\mathbb{P}^N,\mathcal{O}_{\mathbb{P}^N}(d))$
is equal to $V_d( \{ \mathcal{T}_1
\})$. This proves $\dim (\mathbb{P}_d( \{ \mathcal{T}_1 \})= { N+d-1
\choose d }$ and for $X\in \mathbb{P}_d(\{ \mathcal{T}_1 \})$
general, one has $X\cap \Pi_1=C_1$. This finishes the proof of (i) in
case $e=1$. Moreover, we conclude that each $\{ \mathcal{T}_1 \}$ is
suited.

Now assume $e>1$ and assume $\mathcal{L}$ is suited and let
$\mathcal{L}'= \{ \mathcal{T}_1, \ldots ,\mathcal{T}_{e-1} \}$.
Since $\mathcal{L}'$ is suited we can assume (i) holds for
$\mathcal{L}'$ (using the induction hypothesis). First assume $e-1>d$, hence
$\dim (\mathbb{P}_d(\mathcal{L}'))=0$. Since
$\mathbb{P}_d(\mathcal{L})\neq \emptyset$ and
$\mathbb{P}_d(\mathcal{L})\subset \mathbb{P}_d(\mathcal{L}')$, it
follows $\dim (\mathbb{P}_d(\mathcal{L}))=0$, hence (i) holds for
$\mathcal{L}$. So assume $e-1\leq d$, hence
\begin{displaymath}
\dim (\mathbb{P}_d(\mathcal{L}')= { d-e+1+N \choose N }.
\end{displaymath}
In particular (iii) holds for $\mathcal{L}'$ and we apply it to $\Pi
=\Pi_e$. It follows that $C_e\in
\mathbb{P}_{\Pi_e,d}(\mathcal{L}')$, hence $C_e$ defines a
one-dimensional subspace of $V_{\Pi_e,d}(\mathcal{L}')$. The
inverse image under $r$ is exactly $V_d(\mathcal{L})\subset
V_d(\mathcal{L}')$. Since $$\pi_1. \cdots .\pi_e.\Gamma
(\mathbb{P}^N,\mathcal{O}_{\mathbb{P}^N}(d-e))=\ker(r),$$ it follows that
$\ker(r)=0$ if $e=d+1$. This implies $\dim V_d(\mathcal{L})=1$ in
that case, hence $\dim (\mathbb{P}_d(\mathcal{L}))=0$. In case
$e\leq d$, we find $\dim (\ker (r))= { d-e+N \choose N }$ and so $\dim
(V_d(\mathcal{L}))= { d-e+N \choose N }+1$.

Finally, in case $e\leq d$, we are going to prove that a general
$X\in \mathbb{P}_d(\mathcal{L})$ is smooth. We know that $\Pi_1 +
\ldots + \Pi_e+\mathbb{P}_{d-e}\subset \mathbb{P}_d(\mathcal{L})$. For $$X\in
\mathbb{P}_d(\mathcal{L})\setminus (\Pi_1 + \ldots + \Pi_e +
\mathbb{P}_{d-e}),$$ one has $X\cap \Pi_i=C_i$ for $1\leq i\leq e$.
(Indeed, assume there exists $1\leq i\leq e$ such that $\Pi_i\subset
X$. Then for $j\neq i$, consider $\Pi_i\cap \Pi_j\subset X$. But
$C_j\subset \Pi_j\cap X$ and $C_j$ is a good $P$-cone of degree $d$
and $P_j\notin \Pi_i$, so it follows $\Pi_j\subset X$, hence $X\in
\Pi_1 + \ldots + \Pi_e + \mathbb{P}_{d-e}$.) Hence the fixed locus
of $\mathbb{P}_d(\mathcal{L})$ is equal to $C_1\cup \ldots \cup
C_e$. From Bertini's Theorem it follows that $\Sing (X)\subset C_1
\cup \ldots \cup C_e$ for general $X\in \mathbb{P}_d(\mathcal{L})$.
Since for $X\notin \Pi_1 + \ldots + \Pi_e + \mathbb{P}_{d-e}$, one
has $\Pi_i\cap X=C_i$, so it follows that $\Sing(X)\cap C_i\subset \{
P_i \}$. Therefore, if a general $X\in \mathbb{P}_d(\mathcal{L})$
would be singular, then there exists some $1\leq i\leq e$ such that
$X$ is singular at $P_i$. This implies $P_i$ is a singular point of
all $X\in \mathbb{P}_d(\mathcal{L})$. However, inside $\Pi_1 +
\ldots + \Pi_e + \mathbb{P}_{d-e}$, there exists elements of
$\mathbb{P}_d(\mathcal{L})$ smooth at $P_i$ for all $1\leq i\leq e$
(here we used $e\leq d$). Hence a general $X\in
\mathbb{P}_d(\mathcal{L})$ is smooth.
\end{proof}

To finish this section, we mention an indirect characterization of star points that is
already mentioned in the old paper \cite{Shr}. We start by
recalling the following classical definition.

\begin{definition}\label{def5}
Let $X$ be a hypersurface of degree $d$ in $\mathbb{P}^r$ with
equation $F=0$ and let $P=(x_0: \ldots :x_r)$ be a point in
$\mathbb{P}^r$. The \textit{polar hypersurface} of $P$ with respect to $X$ is
the hypersurface $\Delta_P(X)$ defined by the equation
$\sum^r_{i=0}x_i(\partial F/\partial X_i)=0$. Clearly $Q\in X\cap \Delta_P(X)$ with $Q\neq P$ if and only if $\langle
P,Q\rangle \subset T_Q(X)$.
\end{definition}

The following easy lemma shows how polar hypersurfaces can be used to
find star points. Although the proof can be given using equations,
we prefer to give geometric arguments.

\begin{lemma}\label{lem3}
Let $X$ be a smooth hypersurface in $\mathbb{P}^N$ and let $P\in
\mathbb{P}^N$, then $\Delta_P(X)$ contains a hyperplane $\Pi$ of
$\mathbb{P}^N$ containing $P$ if and only if $P$ is a star point of
$X$ and $\Pi =T_P(X)$.
\end{lemma}

\begin{proof}
In case $P$ is a star point on $X$ and $Q\in T_P(X)\cap X$ different
from $P$, then $\langle P,Q\rangle \subset X$. Hence $\langle
P,Q\rangle\subset T_Q(X)$ and so $Q\in \Delta_P(X)$. This proves
$$T_P(X)\cap X\subset \Delta_P(X)\cap T_P(X).$$ If $T_P(X)\not\subset \Delta_P(X)$, we have $\deg (T_P(X)\cap X)=d$
and $\deg (\Delta_P(X)\cap T_P(X))=d-1$, a contradiction. We conclude $T_P(X)\subset \Delta_P(X)$.

Conversely, assume there exists a hyperplane $\Pi$ in $\mathbb{P}^r$
containing $P$ and satisfying $\Pi\subset \Delta_P(X)$. If $Q$ is a
singular point of $\Pi \cap X$, then $\Pi =T_Q(X)$. Since $X$ is
smooth, the equality holds at most for finitely many points (see e.g.
\cite{Zak}*{Corollary 2.8}). Therefore $\Pi \cap X$ has no
multiple components. For each component $Y$ of $\Pi\cap X$ and $Q\in
Y$ general, one finds $\langle P,Q\rangle$ is tangent to $Y$ at $Q$.
Hence the projection with center $P$ in $\Pi$ restricted to $Y$ does not
have an injective tangent map. Because of Sard's Lemma, this is only
possible in case this projection has fibers of dimension at least
one. So we conclude that $Y$ is a cone with vertex $P$. It follows
that $P$ is a star point on $X$.
\end{proof}

\begin{example}\label{exa2}
Consider the Fermat hypersurface $X_{d,N}=Z(X_0^d + \ldots +
X_N^d)=Z(F_{d,N})\subset \mathbb{P}^N$. Let $\xi\in \mathbb{C}$ with
$\xi^d=-1$ and let $E_{i,j}(\xi)$ be the point with $x_i=1$, $x_j=\xi$
and $x_k=0$ for $k\neq i,j$. Clearly, $E_{i,j}(\xi)\in X_{d,N}$. We
are going to show that $E_{i,j}(\xi)$ is a star point on $X_{d,N}$.

From $\partial F_{d,N}/\partial X_k=d.X_k^{d-1}$, we find $(\partial
F_{d,N}/\partial X_k)(E_{i,j}(\xi))=0$ in case $k\notin \{ i,j \}$,
$(\partial F_{d,N}/\partial X_i)(E_{i,j}(\xi))=d$ and $(\partial
F_{d,N}/\partial X_j)(E_{i,j}(\xi))=d.\xi^{d-1}=-d.\xi^{-1}$, hence
$T_{E_{i,j}(\xi)}(X)$ has equation
$d(X_i-1)-d.\xi^{-1}(X_j-\xi)=0$, hence $X_i-\xi^{-1}X_j=0$. In order to compute
$T_{E_{i,j}(\xi)}(X)\cap X$, we replace $X_i$ by $\xi^{-1}X_k$ in the equation of $X_{d,N}$ and one
finds $\sum^N_{k=0; k\notin \{ i,j \}}X^d_k=0$. Inside
$T_{E_{i,j}(\xi)}(X)$, this is a cone with vertex $E_{i,j}(\xi)$. So we find
smooth hypersurfaces in $\mathbb{P}^N$ containing a lot
of star points.

Using Lemma \ref{lem3}, we can prove that we found all
star points on $X_{d,N}$. Indeed for $P=(x_0: \ldots :x_N)$, one has
that the polar hypersurface $\Delta_P(X_{d,N})$ has equation $\sum^N_{i=0}x_iX_i^{d-1}=0$. Clearly,
$$\Sing (\Delta_P(X_{d,N}))=Z(X_k\,|\,x_k\neq 0).$$ If $P$ is a star point then
$\Delta_P(X_{d,N})$ should contain a hyperplane, therefore the singular locus of $\Delta_P(X_{d,N})$ has dimension $N-2$. This implies that there exist $0\leq i<j\leq N$
such that $x_k=0$ for $k\notin \{ i,j \}$. Since there is no point
on $X_{d,N}$ having $N$ coordinates equal to $0$ we can assume
$x_i=1$. Since $P\in X$ we need $x_j^d=-1$. So, the hypersurface $X_{d,N}$ has exactly $d{N+1\choose 2}$ star points.
\end{example}

\begin{remark}
On a smooth cubic surface $X$ in $\mathbb{P}^3$, there are exactly $27$ lines. Each of these lines contains at most two star points (see Proposition \ref{prop2}) and each star point gives rise to three lines, hence the number of star points on such a surface $X$ is at most $18$. This upper bound is attained by the Fermat surface $X_{3,3}\subset \mathbb{P}^3$.
\end{remark}

\section{Collinear star points} \label{section collinear star points}

Let $X$ be a smooth hypersurface in $\mathbb{P}^N$ and let $P_1,
P_2$ be two different star points on $X$. Assume $P_2\in
T_{P_1}(X)$. Since $T_{P_1}(X)\cap X$ is a cone with vertex $P_1$,
it follows that the line $L=\langle P_1, P_2\rangle$ is contained in
$X$. We investigate star points on $X$ belonging to a line
$L\subset X$.

\begin{proposition}\label{prop2}
Let $X$ be a smooth hypersurface of degree $d\geq 3$ in
$\mathbb{P}^N$ and let $L$ be a line in $\mathbb{P}^N$ such that
$L\subset X$. Then $X$ has at most two star points on the line $L$.
\end{proposition}
\begin{proof}
From the assumptions, it follows that $N\geq 3$.

We first consider the case $N=3$, thus $L$ is a divisor on $X$. For
a hyperplane $\Pi$ in $\mathbb{P}^3$, we write $\Pi\cap X$ to denote the
effective divisor on $X$ corresponding to $\Pi$ (hence it is the
intersection considered as a scheme). We consider the pencil
$\mathbb{P}=\{ \Pi\cap X-L\,|\,\Pi \mbox{ is a plane in } \mathbb{P}^3 \mbox{
containing } L \}$ on $X$.

Assume there exist $D\in \mathbb{P}$ with $L\subset D$, hence there
is a plane $\Pi\subset \mathbb{P}^3$ with $\Pi\cap X\geq 2L$. It follows
that $\Pi=T_P(X)$ for all $P\in L$. If $X$ would have a star point $P$
on $L$ then any $Q\in L\setminus \{ P \}$ would be a singular point
of $T_P(X)\cap X$, this is impossible because of Lemma \ref{lem1}.
So, since for each $D\in \mathbb{P}$ one has $L\not\subset D$, the pencil
$\mathbb{P}$ induces a $g^1_{d-1}$ on $L$.

If $P\in L$ is a star point of $X$, then there is a divisor $L_1 +
\ldots + L_{d-1}$ in $\mathbb{P}$ with $L_1, \ldots, L_{d-1}$ lines
through $P$ different from $L$, hence $(d-1)P\in g^1_{d-1}$. Since
$g^1_{d-1}$ on $\mathbb{P}^1$ can have at most two total ramification
points, it follows that $X$ has at most two star points on $L$. (Note
that here we use $d\geq 3$. Indeed, this lemma clearly does not
hold for quadrics in $\mathbb{P}^3$).

Now assume $N>3$ and assume the proposition holds in
$\mathbb{P}^{N-1}$. Consider the linear system $\mathbb{P}= \{ X\cap
\Pi\,|\, \Pi \mbox{ is a hyperplane in } \mathbb{P}^N \mbox{ containing } L
\}$ in $\mathbb{P}^N$. It has dimension $N-2$ and its fixed locus is
$L$. It follows from Bertini's Theorem that for a general $D\in
\mathbb{P}$ one has $\Sing (D)\subset L$. However, in case $P\in L$
is a singular point of $D\in \mathbb{P}$, then $D=X\cap T_P(X)$, hence
the locus of divisors $D\in \mathbb{P}$ singular at some point of
$L$ has at most dimension one. Since $N-2>1$, it follows that a general $D\in
\mathbb{P}$ is smooth.

So take a general hyperplane $\Pi$ in $\mathbb{P}^N$ containing $L$
and consider $$X'=X\cap \Pi \subset \Pi\cong \mathbb{P}^{N-1},$$
again a smooth hypersurface. If $P$ is a star point for $X$ on $L$,
then by definition it is also a star point for $X'$. By induction,
we know $X'$ has at most two star points on $L$, hence $X$
has at most two star points on $L$.
\end{proof}

\begin{example}
The affine surface $X$ in $\mathbb{A}^3$ with equation
$$yx(y-x)(y+x)+z(x-1)(z+x-1)(z-x+1)=0$$ has degree equal to $4$ and contains the $x$-axis.
On this axis, there are two star points, namely $P_1=(0,0,0)$ with $T_{P_1}X:z=0$ and $P_2=(1,0,0)$ with $T_{P_2}X:y=0$.

\begin{figure}[htp]
\centering
\includegraphics[height=5cm]{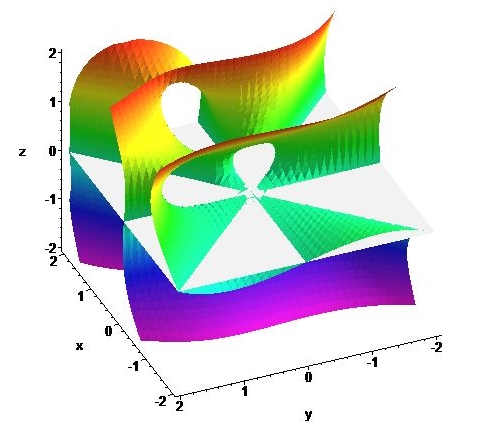}
\vspace{-3mm}
\caption{Star point $P_1$}\label{figure}
\vspace{-2mm}
\end{figure}
On Figure \ref{figure}, one can see that the lines $L_1:y=0$ (i.e. the $x$-axis), $L_2:x=0$, $L_3:y=x$ and $L_4:y=-x$ on $T_{P_1}X$ are contained in $X$.
\end{example}

Now we consider star points on a line $L$ not contained in the
smooth hypersurface $X$. Of course, such a line $L$ contains at most
$d$ star points of $X$. The following result shows this upper bound occurs.

\begin{proposition}\label{prop3}
Fix integers $d\geq 3$ and $N\geq 3$. Let $L$ be a line in
$\mathbb{P}^N$ and let $P_1, \ldots, P_d$ be different points on
$L$. Fix hyperplanes $\Pi_1, \ldots, \Pi_d$ not containing $L$ and such
that $P_i\in \Pi_i$ for $1\leq i\leq d$ and fix a good $P_1$-cone
$C_1$ in $\Pi_1$ of degree $d$. Then there exists a smooth
hypersurface $X$ of degree $d$ in $\mathbb{P}^N$ not containing $L$,
having a star point at $P_1, \ldots, P_d$ with $T_{P_i}(X)=\Pi_i$ for
$1\leq i\leq d$ and $T_{P_1}(X)\cap X=C_1$.
\end{proposition}

\begin{proof}
Let $Y$ be a cone on $C_1$ with vertex in $L$ different from $P_1,\ldots,P_d$. For $2\leq i\leq d$, let $Y\cap \Pi_i=C_i$. Clearly, $C_i$ is a good $P_i$-cone in $\Pi_i$. Let $\mathcal{T}_i=(\Pi_i,P_i,C_i)$ for $1\leq i\leq d$ and denote $\mathcal{L}=\{\mathcal{T}_1,\ldots,\mathcal{T}_d\}$. Since $Y\in\mathbb{P}_d(\mathcal{L})$ and $\Pi_i\not\subset Y$ for $1\leq i\leq d$, we find $\mathcal{L}$ is suited for degree $d$. From Theorem \ref{theorem main}, it follows a general element $X$ of $\mathbb{P}_d(\mathcal{L})$ is smooth.
\end{proof}

For smooth cubic plane curves $C$, there is the following classical
result: if $P_1, P_2$ are
inflection points of $C$, then the third intersection point of $C$
with the line connecting $P_1$ and $P_2$ is also an inflection
point. For plane curves, this result has a generalization. Indeed, let
$C$ be a smooth plane curve of degree $d\geq 2$, let $L$ be a line
and assume $P_1, \ldots, P_{d-1}$ are total inflection points of $C$
contained in $L$. Then $C$ has one more intersection point $P_d$
with $L$ and $P_d$ is also a total inflection point of $C$. We
generalize this to the case of star points. This result is an extra indication
that the concept of star point is the correct generalization of
the concept of total inflection point.

\begin{theorem}\label{theo1}
Let $N\geq 3$ be an integer and let $X$ be a smooth hypersurface of
degree $d\geq 3$ in $\mathbb{P}^N$. Let $L$ be a line in
$\mathbb{P}^N$ with $L\not\subset X$ and assume $P_1, \ldots,
P_{d-1}$ are $d-1$ different star points of $X$ on $L$. Then $L$
intersects $X$ transversally at $d$ points $P_1, \ldots, P_d$ and
$P_d$ is also a star point of $X$.
\end{theorem}

\begin{proof}
Since $T_{P_i}(X)\cap X$ is a good $P_i$-cone $C_i$ and $L\not\subset X$, it follows that
$L\not\subset T_{P_i}(X)$. In particular, $L$ intersects $X$
transversally at $P_i$ for $1\leq i\leq d-1$, hence $L$ intersects
$X$ at $d$ different points $P_1, \ldots, P_d$. Let $\Pi_i=T_{P_i}(X)$ for $1\leq i\leq d$.

We need to prove that $P_d$ is a star point of $X$. Since $L\not\subset \Pi_d$ it follows that $\Pi_d\cap
\Pi_i$ is a hyperplane in $\Pi_i$ not containing the vertex of the
$P_i$-good cone $C_i$. This implies $C_{i,d}=C_i\cap \Pi_d$ is a
smooth hypersurface of degree d in $\Pi_i\cap \Pi_d$. Take $i=1$ and let
$P\in C_{1,d}$ and consider the plane $\Lambda _P=\langle
P,L\rangle$ spanned by $P$ and $L$. Since $L\not\subset \Pi_d$ and
$P_d, P\in \Pi_d$, it follows that the line $\langle P_d, P\rangle$ is
equal to $\Lambda _P\cap \Pi_d$. Since $P\in X\cap \Pi_1$, it follows
that the line $\langle P,P_1\rangle$ belongs to $X$, hence $\langle
P,P_1\rangle \subset \Lambda _P\cap X$. Let $Q_i=\langle P_1,
P\rangle \cap \Pi_i$ for $2\leq i\leq d-1$, then $Q_i\in X\cap \Pi_i$
while $P_i\neq Q_i$, hence $\langle P_i, Q_i\rangle \subset
C_i=X\cap \Pi_i$ and $\langle P_i, Q_i\rangle \subset \Lambda_P\cap
X$. So we obtain $d-1$ different lines on $X\cap \Lambda_P$, hence
\begin{displaymath}
\langle P_1, P\rangle + \langle P_2,Q_2 \rangle + \ldots + \langle
P_{d-1}, Q_{d-1} \rangle \subset \Lambda _P\cap X
\end{displaymath}
Since $\Lambda _P\not\subset X$, we know that $\Lambda _P\cap X$ is a
curve (effective divisor) of degree $d$ on $\Lambda _P$. Hence
$\Lambda _P\cap X$ needs to be a sum of $d$ lines in $\Lambda _P$.
Since $$P_d\not\in \langle P_1, P\rangle \cup \langle P_2, Q_2
\rangle \cup \ldots \cup \langle P_{d-1}, Q_{d-1} \rangle,$$ it
follows that $\Lambda _P\cap X$ contains a line $T$ through $P_d$.
Since $T\subset X$, it follows that $T\subset \Pi_d$, hence $T\subset
\Pi_d\cap \Lambda_P$ and therefore $T= \langle P_d, P\rangle$. Since
$P$ is any point on $C_{1,d}$ it follows that the cone $C_d$ on
$C_{1,d}$ with vertex $P_d$ is contained in $\Pi_d\cap X$. However
this cone is a hypersurface of degree $d$ in $\Pi_d\cap X$, hence
$C_d=\Pi_d\cap X$. Since $C_d$ is a good $P_d$-cone in $\Pi_d$, it
follows that $P_d$ is a star point on $X$.
\end{proof}

\section{The number of star points} \label{section number}

The number of total inflection points on a smooth plane curve is
bounded because a total inflection point gives some contribution to
the divisor corresponding to the inflection points of the associated
two-dimensional linear system. The degree of this divisor is fixed by
the degree of the plane curve. For the case of star points on
hypersurfaces of dimension at least two, such argument is not
available.

Although there exist hypersurfaces having a large
number of star points (see Proposition \ref{prop3} or Example \ref{exa2}),
we prove the finiteness of the number of
star points on a smooth hypersurface. First we prove that the locus
of star points form a Zariski-closed subset. This implies that a
smooth hypersurfaces $X$ having infinitely many star points should contain
a curve $\Gamma$ such that each point on that curve is a star point
of $X$.

\begin{lemma}\label{lem2}
Let $X\subset \mathbb{P}^N$ be a smooth hypersurface of degree
$d\geq 3$ and let $ST(X)$ be the set of star points on $X$. Then the set
$ST(X)$ is a Zariski-closed subset of $X$.
\end{lemma}

\begin{proof}
Let $(\mathbb{P}^N)^{\ast}$ be the dual space of $\mathbb{P}^N$
(parameterizing hyperplanes in $\mathbb{P}^N$) and let
$\mathcal{H}\subset (\mathbb{P}^N)^{\ast}\times \mathbb{P}^N$ be the
incidence space (as a set it is defined by $(\Pi,P)\in \mathcal{H}$ if
and only if $P\in \Pi$; clearly it is Zariski-closed in
$(\mathbb{P}^N)^{\ast}\times \mathbb{P}^N$). Consider also the
Zariski-closed subset $\widetilde{X}\subset X\times (\mathbb{P}^N)^{\ast}$ (isomorphic to $X$) with $(P,\Pi)\in
\widetilde{X}$ if and only if $\Pi=T_P(X)$. Then
$\mathbb{P}(T_X)$ can be identified with $$(\widetilde{X}\times
\mathbb{P}^N)\cap (X\times \mathcal{H})\subset X\times (\mathbb{P}^N)^{\ast}\times \mathbb{P}^N.$$ The projection
$p:\mathbb{P}(T_X)\rightarrow X$ is a $\mathbb{P}^{N-1}$-bundle and
$\mathcal{X}=(\widetilde{X}\times X)\cap (X\times \mathcal{H})$ is a
divisor in $\mathbb{P}(T_X)$ giving a family $p':\mathcal{X}\rightarrow X$ of hypersurfaces of
degree $d$ for $p$. There is a
natural section $s:X\rightarrow \mathcal{X}$ with
$s(P)=(P,T_P(X),P)$. Using local equations, it is clear that the set
of points $Q$ on $\mathcal{X}$ such that $p'^{-1}(p'(Q))$ has
multiplicity $d$ at $Q$ is a Zariski-closed subset $Y$ of
$\mathcal{X}$. Since $p'(Y\cap s(X))=ST(X)$ and $p'$ is
proper, it follows that $ST(X)$ is Zariski-closed in $X$.
\end{proof}

The proof of the following theorem is inspired by \cite{LaTo}. In our situation, it is possible to work a bit more geometrically.

\begin{theorem} \label{theo finite}
Let $X$ be a smooth hypersurface of degree $d\geq 3$ in
$\mathbb{P}^n$, then $X$ has at most finitely many star points.
\end{theorem}

\begin{proof}
Assume $X$ has infinitely many star points, then there exists a
curve $\Gamma$ on $X$ such that each $P\in \Gamma$ is a star point
on $X$, so $C_P=X\cap T_P(X)$ is a cone with vertex $P$.

Let $\mathbb{G}=\mathbb{G}(1,N)$ be the Grassmannian of lines in $\mathbb{P}^N$. For a line $L\subset \mathbb{P}^N$, we will denote the corresponding point in $\mathbb{G}$ by $l$. For each star point $P\in \Gamma$, the set of lines in $X$ through $P$ gives rise to a subset $\mathcal{C}_P\subset \mathbb{G}$ of dimension $N-3$. By moving $P$ on $\Gamma$ and $\mathcal{C}_P$ inside
$\mathbb{G}$, we obtain a set $B\subset \mathbb{G}$ with $\dim(B)=N-2$.

Let $P$ be a general point on $\Gamma$, pick $l$ general in $\mathcal{C}_P$ and let $Q$ be a general point on $L$. Choose coordinates $(X_0:\ldots:X_N)$ on $\mathbb{P}^N$ such that $P=E_0$, $Q=E_1$ and $T_P(X):X_2=0$, where $E_i$ is the point with zero coordinates except the $i$th coordinate being one. Write $T=\cap_{R\in L} T_R(X)$. Note that $T$ is a linear space of dimension equal to $N-2$ or $N-1$, since $T_Q(C_P)\subset T\subset T_Q(X)$.

Assume $\dim(T)=N-1$. In this case, for each $R\in L$, the tangent space $T_R(X)$ has equation $X_2=0$, thus $\frac{\partial F}{\partial X_i}(R)=0$ for all $i\neq 2$. On the other hand, there is at least one point $R^{\star}\in L$ such that $\frac{\partial F}{\partial X_2}(R^{\star})=0$. This implies $R^{\star}$ is a singular point of $X$, a contradiction.

Now assume $\dim(T)=N-2$. Note that $T_P(\Gamma)\not\subset T$. Indeed, otherwise we have $T_P(\Gamma)\subset T_S(C_P)$ for all $S\in C_P$, hence,  from Sard's Lemma, it follows that the projection of $C_P$ from $T_P(\Gamma)$ has $2$-dimensional fibers. Because it is a projection, those fibers are planes containing $T_P(\Gamma)$ and $C_P$ is a cone with vertex $T_P(\Gamma)$. Since $P$ is the only singular point of $C_P$, this is impossible.

So we can choose the coordinates on $\mathbb{P}^N$ so that $T_P(\Gamma)=\langle P,E_3\rangle$ and $T_Q(X)$ has equation $X_3=0$, thus $T$ has equation $X_2=X_3=0$. For each $R\in L$, we have $\frac{\partial F}{\partial X_i}(R)=0$ for all $i\not\in\{2,3\}$ and the tangent space $T_R(X)$ has equation $\frac{\partial F}{\partial X_2}(R)X_2+\frac{\partial F}{\partial X_3}(R)X_3=0$.

Let $\mathbb{P}^N=\mathbb{P}(W)$ for some $(N+1)$-dimensional vector space $W$ and let $\{e_0,\ldots,e_N\}$ be a basis of $W$ such that $E_i=[e_i]$. Since $T_Q(X)$ has equation $X_3=0$, there exists a holomorphic arc $\{q(t)\}\subset W$ with $Q(t)=[q(t)]\in X$, $q(0)=e_1$ (so $Q=[q(0)]$) and $q'(0)=e_2$. Note that the arc $\{Q(t)\}\subset X$ is not contained in $T_P(X)$. Since the lines in $B$ cover $X$, there exists an arc $\{p(t)\}\subset W$ such that $P(t)=[p(t)]\in \Gamma$, $P(0)=P$ and the lines $\langle P(t),Q(t)\rangle$ are contained in $B$. Write $p'(0)=\lambda.e_3$ with $\lambda \in\mathbb{C}$.

Now let $R=(a:b:0:\ldots:0)$ be general point on $L$, so $\{r(t)\}\subset W$ with $r(t)=a.p(t)+b.q(t)$ is a holomorphic arc with $R(t)=[r(t)]\in X$ and $R(0)=R$. We have $r'(0)=a\lambda.e_3+b.e_2$, hence $$[r(0)+r'(0)]=(a:b:b:a\lambda:0:\ldots:0)\in T_R(X).$$ If $\lambda=0$, we have $\frac{\partial F}{\partial X_2}(R)=0$ for all point $R\in L\setminus\{P\}$, and thus $\frac{\partial F}{\partial X_2}(P)=0$, a contradiction. So $\lambda\neq 0$ and we conclude $\frac{\partial F}{\partial X_2}X_1+\lambda\frac{\partial F}{\partial X_3}X_0=0$ on $L$. This implies the existence of a homogeneous form $G(X_0,X_1)$ of degree $d-2$ such that $\frac{\partial F}{\partial X_2}=\lambda X_0 G$ and $\frac{\partial
F}{\partial X_3}=-X_1 G$ on $L$. Choose $(a^{\star},b^{\star})\neq (0,0)$
such that $G(a^{\star},b^{\star})=0$. This corresponds to a point $R^{\star}\in L$
satisfying $\frac{\partial F}{\partial X_2}(R^{\star})=\frac{\partial F}{\partial
X_3}(R^{\star})=0$. Hence $R^{\star}$ is a singular point on $X$, again a
contradiction.
\end{proof}

\begin{corollary}
Let $X$ be a smooth hypersurface in $\mathbb{P}^N$ of degree $d\geq 3$ and let $\mathbb{F}(X)$ be its Fano scheme of lines. Then there does not exist a $(N-2)$-dimensional subset $B\subset \mathbb{F}(X)$ and a curve $\Gamma\subset X$ such that each line of $B$ meets $\Gamma$.
\end{corollary}
\begin{proof}
Since a smooth surface $X\subset \mathbb{P}^3$ of degree $d\geq 3$ contains only finitely many lines, the statement is clear for $N=3$. Now assume the statement fails for some $N>3$. So, there exists a smooth hypersurface $X\subset \mathbb{P}^N$ of degree $d\geq 3$, a subset $B\subset \mathbb{F}(X)$ of dimension $N-2$ and a curve $\Gamma\subset X$ such that each line of $B$ meets $\Gamma$. We will denote the line corresponding to a point $l\in B$ by $L$. For each point $P\in\Gamma$, let $\mathcal{C_P}=\{l\in B\,|\,P\in L\}$ and let $C_P=\cup_{l\in\mathcal{C}_P}\,L \subset X\cap T_P(X)$. Note that the dimension of $\mathcal{C}_P$ is equal to $N-3$ (since $\dim(B)=N-2$ and $X\neq T_P(X)$), so $C_P$ is a hypersurface in $X\cap T_P(X)$. Let $P$ be a general point of $\Gamma$. If $C_P=X\cap T_P(X)$, the point $P$ is a star point, hence $X$ has infinitely star points. This is in contradiction with Theorem \ref{theo finite}. Now assume $C_P\neq X\cap T_P(X)$. So $T_P(X)\cap X$ has an irreducible component $D_P$ of dimension $N-2$ different from $C_P$. For each point $Q\in C_P\cap D_P$, we have $T_Q(X)=T_P(X)$ since $Q$ is singular in $X\cap T_P(X)$. This is in contradiction with \cite[Corollary 2.8]{Zak}, since $C_P\cap D_P$ has dimension at least one.
\end{proof}

\section{Configurations of star points} \label{section conf}

In Section $3$ of our paper \cite{CC1}, we obtain a lower bound on the dimension of the configuration space for total inflection points on smooth plane curves. This was obtained by describing that configuration space as an intersection of two natural sections of some natural smooth morphism. Now we  generalize this to the case of star points. We use Notation \ref{not2} and introduce some further notations.

\begin{notation} \label{not4}
The set $\mathcal{P}_d^{e,0}$ consists of $e$-tuples $(\Pi_i,P_i,
C_i)_{i=1}^{e}$ from $\mathcal{P}_d$ satisfying $P_i \notin \Pi_j$ for $i\neq j$. We
write $\mathcal{L}_i$ to denote the associated $i$-tuple $(\Pi_j,
P_j,C_j)^{i}_{j=1}$ for $1\leq i\leq e$.

The set $\mathcal{V}_{d,e}\subset \mathcal{P}_d^{e,0}$ is determined by the
condition that $(\Pi_i,P_i,C_i)_{i=1}^{e}\in \mathcal{P}_d^{e,0}$ belongs to $\mathcal{V}_{d,e}$ if
and only if there exists an irreducible hypersurface $X$ of degree
$d$ in $\mathbb{P}^N$ such that $X\cap \Pi_i=C_i$ for $1\leq i\leq
e$.
\end{notation}

\begin{notation}\label{not6}
Given $\mathcal{L}\in \mathcal{P}_d^{e,0}$, we write
$\Pi_{i,j}=\Pi_i \cap \Pi_j$ for $1\leq i<j\leq e$. Associated to
$\mathcal{L}$, we introduce linear systems $g'_i$ on $\Pi_i$ for $1\leq i\leq
e$ as follows. In case $i\leq d+1$, then $g'_i=\Pi_{1,i}+ \ldots +
\Pi_{i-1,i}+\mathbb{P}_{\Pi_i,d-i+1}$ and $g'_i$ is empty in
case $i>d+1$.

Let $\mathcal{G}^{e}$ be the space of pairs $(g,\mathcal{L})$ with
$\mathcal{L}\in \mathcal{P}_d^{e,0}$ and $g$ is an $(e-1)$-tuple
$(g_2,\ldots,g_e)$ of linear systems $g_i$ on $\Pi_i$ as follows. In case $i\leq
d+1$, then $g_i$ is a linear subsystem of $\mathbb{P}_{\Pi_i,d}$ of
dimension ${ N+d-i \choose N-1 }$ containing $g'_i$ and in case
$i>d+1$, $g_i$ has dimension $0$ (i.e. it consists of a unique
effective divisor). We write $\tau :\mathcal{G}^{e}\rightarrow
\mathcal{P}_d^{e,0}$ to denote the natural projection.
\end{notation}

From the previous definition, it follows that $\tau$ is a smooth
morphism of relative dimension $$f=\sum_{i=2}^{e}\left[ { N+d-1
\choose N-1 }-{ N+d-i \choose N-1 }-1 \right]$$ in case $e\leq
d+1$ and of relative dimension $$f=\sum_{i=2}^{d+1}\left[ {
N+d-1 \choose N-1 }-{ N+d-i \choose N-1 }-1 \right] + (e-d-1)\left[{
N+d-1 \choose N-1 }-1\right]$$ in case $e>d+1$.

In the proof of the next theorem we are going to use the following
definition.

\begin{definition}\label{def4}
For a point $P$ on a hyperplane $\Pi \subset \mathbb{P}^N$ and an
integer $k\geq 1$, we write $k(P\in \Pi )$ to denote the fat point on
$\Pi$ located at $P$ of multiplicity $k$. So,
$k(P\in \Pi )$ is the $0$-dimensional subscheme of $\Pi$ of length ${
N+k-2 \choose N-1 }$ with support $P$ and locally defined at $P$ by
the ideal $\mathcal{M}^k_{P,\Pi}$.
\end{definition}

\begin{theorem}\label{theo2}
There exist two sections $S_1, S_2$ of $\tau$ such that
$\mathcal{V}_{d,e}=\tau (S_1\cap S_2)$.
\end{theorem}

\begin{corollary}\label{cor}
Each irreducible component of $\mathcal{V}_{d,e}$ has codimension at least $f$
inside $\mathcal{P}_d^{e,0}$.
\end{corollary}

\begin{proof}[Proof of theorem \ref{theo2}]
For $\mathcal{L}\in \mathcal{P}_d^{e,0}$, we define
$S_1(\mathcal{L})=g$ as follows. In case $2\leq i\leq d+1$,
$g_i=\langle g'_i, C_i\rangle$ and in case $i>d+1$, then $g_i=\{ C_i
\}$.

Next we construct the section $S_2(\mathcal{L})=h$.
Consider the surjective map
\begin{displaymath}
q_2:V_d(\mathcal{L}_1)\rightarrow V_{\Pi_2,d}(\mathcal{L}_1).
\end{displaymath}
We know that $\ker(q_1)=\pi_1\pi_2V_{d-2}$ (where $\pi_i$ is the equation of $\Pi_i\subset \mathbb{P}^N$).

Take $h_2=\mathbb{P}_{\Pi_2,d}(\mathcal{L}_1)$. We know
$V_{\Pi_2,d}(\mathcal{L}_1)$ contains $\pi_{1,2}V_{\Pi_2,d-1}$ (where $\pi_{i,j}$ is the equation of $\Pi_{i,j}\subset \Pi_j$) and
$\dim (h_2)={ d+N-2 \choose N-1 }$. From the restriction map
$\mathcal{O}_{\Pi_2}(d)\rightarrow \mathcal{O}_{d(P_2\in \Pi_2)}(d)$,
we obtain a homomorphism
\begin{displaymath}
V_{\Pi_2,d}(\mathcal{L}_1)\rightarrow \Gamma (\mathcal{O}_{d(P_2\in
\Pi_2)}(d)).
\end{displaymath}
The image of $\pi_{1,2}V_{\Pi_2,d-1}$ is equal to $\Gamma
(\mathcal{O}_{d(P_2\in \Pi_2)}(d))$, hence the kernel of this map is
a 1-dimensional vector space $\langle s_2\rangle$. Let
$V'_d(\mathcal{L}_2)=q^{-1}(\langle s_2\rangle)$; it is a
subspace of $V_d$ containing $\pi_1\pi_2V_{d-2}$. Moreover, the
associated linear system $\mathbb{P}(V'_d(\mathcal{L}_2))$ has
dimension ${ N+d-2 \choose N }$ and it induces a unique divisor on
$\Pi_2$ with multiplicity at least $d$ at $P_2$.

Let $2<i\leq e$ with $i\leq d+1$ and assume $h_j$ is constructed for $2\leq j\leq
i-1$ and assume $V'_d(\mathcal{L}_j)\subset V_d$ is constructed such
that $V'_d(\mathcal{L}_j)$ contains $\pi_1. \cdots .\pi_jV_{d-j}$,
such that $V'_d(\mathcal{L}_j)$ is contained in $V'_d(\mathcal{L}_{j-1})$ (with
$V'_d(\mathcal{L}_1)=V_d(\mathcal{L}_1)$)
and the associated linear system $\mathbb{P}(V'_d(\mathcal{L}_j))$
has dimension $ { N+d-j \choose N }$. Assume also that it induces a unique
divisor on $\Pi_j$ with multiplicity at least
$d-j+2$ at $P_j$. The restriction of forms of
degree $d$ to $\Pi_i$ gives rise to a map
\begin{displaymath}
q_i : V'_d(\mathcal{L}_{i-1}) \rightarrow V_{\Pi_i,d}=\Gamma (\Pi_i,
\mathcal{O}_{\Pi_i}(d)).
\end{displaymath}
Assume $s\in \ker (q_i)$, i.e. $\Pi_i\subset Z(s)$. Since $s\in
V_d(\mathcal{L}_1)$, we obtain $Z(s)\cap \Pi_1$ contains $C_1\cup
\Pi_{1,i}$. Since $P_1\notin \Pi_{1,i}$, we find $\Pi_1\subset Z(s)$.
Let $1<i_0\leq i-1$ and assume $\Pi_1 \cup \ldots \cup
\Pi_{i_0-1}\subset Z(s)$, then $Z(s)\cap \Pi_{i_0}$ contains
$$\Pi_{1,i_0}\cup \ldots \cup \Pi_{i_{0}-1,i_0}\cup \Pi_{i_0,i}.$$ If
$Z(s)$ does not contain $\Pi_{i_0}$, then $Z(s)\cap \Pi_{i_0}$ is a
divisor in $\Pi_{i_0}$ having multiplicity at most $d-i_0$ at
$P_{i_0}$. But $s\in V'_d(\mathcal{L}_{i_0})$, hence we obtain a
contradiction and therefore $Z(s)$ contains $\Pi_{i_0}$. In this way
we find
\begin{displaymath}
\ker (q_i)=\pi_1 . \cdots .\pi_iV_{d-i}
\end{displaymath}
(hence this kernel is empty in case $i=d+1$). This implies
\begin{displaymath}
\dim (\im q_i)= { N+d-i+1 \choose N }+1- { N+d-i \choose N } = {
N+d-i \choose N-1 } +1.
\end{displaymath}
Since $\pi_1 . \cdots . \pi_{i-1}V_{d-i+1}\subset
V'_d(\mathcal{L}_{i-1})$, one has $$\pi_{1,i} . \cdots .
\pi_{i-1,i}V_{\Pi_i,d-i+1}\subset \im (q_i),$$ so we take
$h_i=\mathbb{P}(\im q_i)$.

As before, one concludes that the restriction map $$\im
(q_i)\rightarrow \Gamma (\mathcal{O}_{(d-i+2)(P_i\in \Pi_i)})$$ is
surjective, hence it has a one-dimensional kernel $\langle s_i
\rangle$. Let $V'_d(\mathcal{L}_i)=q_i^{-1}(\langle s_i \rangle)$,
then the inclusion $V'_d(\mathcal{L}_i)\subset V'_d(\mathcal{L}_{i-1})$ holds,
$V'_d(\mathcal{L}_i)$ contains $\pi_1. \cdots .\pi_iV_{d-i}$, the
associated linear system $\mathbb{P}(V'_d(\mathcal{L}_i))$ has
dimension ${ N+d-i \choose N }$ and it induces a unique divisor on
$\Pi_i$ with multiplicity at least $d-i+2$ at $P_i$.

In case $i=d+1$ (hence $e\geq d+1$), then
$\mathbb{P}(V'_d(\mathcal{L}_{d+1}))$ consists of a unique divisor
$D$. If $D$ contains $\Pi_i$ for some $i<d+1$ then as before one
proves $D$ contains $\Pi_1, \ldots, \Pi_{d+1}$, hence a
contradiction. So for $i>d+1$, one takes $h_i=\{ D \}$.

Clearly $S_1(\mathcal{L})=S_2(\mathcal{L})$ if and only if
$g_i=h_i$ for $2\leq i\leq e$. The equality $h_2=g_2$ is equivalent to $C_2\in
\mathbb{P}_{\Pi_2,d}(\mathcal{L}_1)$, which is equivalent to $\mathcal{L}_2\in \mathcal{V}_{d,2}$ by Theorem
\ref{theorem main}.
In this case, we also have $V'_d(\mathcal{L}_2)=V_d(\mathcal{L}_2)$.

Let $3\leq i\leq e$ and assume that $h_j=g_j$ for $2\leq j\leq i-1$ is equivalent to $\mathcal{L}_{i-1}\in \mathcal{V}_{d,i-1}$ and
in this case, $V'_d(\mathcal{L}_{i-1})=V_d(\mathcal{L}_{i-1})$. Assume $h_j=g_j$ for $2\leq j\leq i$.
Hence, $\mathcal{L}_{i-1}\in \mathcal{V}_{d,i-1}$ and the image of $q_i$ is equal to $V_{\Pi_i,d}(\mathcal{L}_{i-1})$. From the previous arguments, it
follows that $\mathbb{P}(\im (q_i))=h_i$ contains a unique divisor
having multiplicity at least $\max \{ d-i+2, 0 \}$ at $P_i$, hence
$h_i=g_i$ is equivalent to $C_i\in \mathbb{P}(\im
q_i)=\mathbb{P}_{\Pi_i,d}(\mathcal{L}_i)$. This is equivalent to
$\mathcal{L}_i\in \mathcal{V}_{d,i}$ and by construction
$V'_d(\mathcal{L}_i)=V_d(\mathcal{L}_i)$ in that case.

By induction on $i$, we conclude $S_1(\mathcal{L})=S_2(\mathcal{L})$ if and only if $\mathcal{L}\in \mathcal{V}_{d,e}$.
\end{proof}

\begin{remark}
Take $N=3$, hence $\dim (\mathcal{P}^{e,0}_d)=e(d+5)$.

In case $e=2$, we find $f=d$, hence $\dim (\mathcal{V}_{d,2})\geq d+10$.
Combining this with Theorem \ref{theorem main}, the space of surfaces in
$\mathbb{P}^3$ having at least $2$ star points not contained in a line inside the surface has dimension at least
$d+10+ { d+1 \choose 3 }$. In case $d=3$, we find this space has
dimension at least $17$. In his thesis \cite{Ngu1} (see also
\cite{Ngu2}), Nguyen found exactly two components both having dimension
$17$.

In case $e=3$, we find $f=3d$, hence $\dim(\mathcal{V}_{d,3})\geq 15$. In case
$\mathcal{V}_{d,3}$ would have a component of dimension exactly $15$, it
would be an orbit under the action of $\Aut (\mathbb{P}^3)$.
Combining this description with Theorem \ref{theorem main}, the space of
hypersurfaces with three star points has dimension at least $15+ { d
\choose 3 }$. In case $d=3$, we find this space has dimension at
least $16$. In his work, Nguyen found a unique component of dimension $17$.
This component comes from the fact that a cubic surface $X$
in $\mathbb{P}^3$ having two star points $P_1$ and $P_2$, such that
the line connecting those points is not contained in $X$, has at
least three star points (see Theorem \ref{theo1}). It follows that the space of cubic surfaces with
at least three star points not on a line has only irreducible components
of dimension 16 (but those components do parameterize surface having
more than three star points because of Theorem \ref{theo1}).
\end{remark}

The following proposition is a generalization of a result on total
inflection points on smooth plane curves proved in \cite{Verm} to
the case of star points. It shows that the case of three star points (see Section \ref{section three starpoints})
and collinear star points (see Section \ref{section collinear star points}) are the most basic cases. The proof given in
\cite{Verm} relies on a very general formulated algebraic statement,
that can also be applied in the situation of star points. Here we
give a more direct and more geometric proof based on Theorem
\ref{theorem main}.

\begin{proposition}\label{prop5}
Let $\mathcal{L}=(\Pi_i, P_i, C_i)_{i=1}^{e}\in
\mathcal{P}_d^{e,0}$. Assume that
\begin{enumerate}
\item[(i)] for all $1\leq i_1< i_2< i_3\leq e$, one has $(\Pi_{i_j}, P_{i_j},
C_{i_j})_{j=1}^3\in \mathcal{V}_{d,3}$;
\item[(ii)] for all $1\leq i_1< i_2< \cdots < i_m\leq e$ with $\dim
(\Pi_{i_1}\cap \Pi_{i_2}\cap \cdots \cap \Pi_{i_m})=N-2$, one has
$(\Pi_{i_j}, P_{i_j}, C_{i_j})_{j=1}^m\in \mathcal{V}_{d,m}$;
\end{enumerate}
then $\mathcal{L}\in \mathcal{V}_{d,e}$.
\end{proposition}

\begin{proof}
Of course we may assume $e>3$. Assume both conditions hold and
assume the conclusion of the proposition holds for $e-1$ instead of
$e$.

In case $\Pi_{i,e}=\Pi_{j,e}$ for all $1\leq i<j\leq e-1$, then
the second assumption (applied to $i_j=j$ for $1\leq j\leq e$) implies
$\mathcal{L}\in \mathcal{V}_{d,e}$.

So we can assume $\Pi_{1,e}\neq \Pi_{2,e}$. Applying the
induction hypothesis to $i\in \{ 1, \cdots, e-1\}$, we find
$X'\in\mathbb{P}_d$ such that $X'\cap \Pi_i=C_i$ for $1\leq i\leq
e-1$. Let $X'\cap \Pi_e= C'_e$. Let $\varphi'_e$ (resp. $\varphi_e$)
be the equation of $C'_e$ (resp. $C_e$) inside $\Pi_e$. For $j\in\{1,2\}$,
consider the linear system of hypersurfaces $X\in \mathbb{P}_d$ such
that $X\cap \Pi_i=C_i$ for $i\in \{j,3,\ldots,e-1\}$. On $\Pi_e$,
this linear subsystem of $\mathbb{P}_d$ induces a linear subsystem
of $\mathbb{P}_{\Pi_e,d}$. From the induction hypothesis applied to
$\{j,3,\ldots,e\}$ this linear subsystem of
$\mathbb{P}_{\Pi_e,d}$ is equal to $$\left\langle
\Pi_{j,e}+\sum_{i=3}^{e-1}\Pi_{i,e}+\mathbb{P}_{\Pi_e,d-e+2},
C_e\right\rangle$$ in case $d-e+2\geq 0$; otherwise it is $\{ C_e\}$.

Since $C'_e$ belongs to those linear subsystems of
$\mathbb{P}_{\Pi_e,d}$, there exists $g_j\in V_{\Pi_e,d-e+2}$ and
$c_j\in \mathbb{C}$ such that
$\varphi'_e=\pi_{j,e}.\prod_{i=3}^{e-1}\pi_{i,e}.g_j+c_j.\varphi_e$
(here we write $\pi_{i,e}$ to denote the equation of $\Pi_{i,e}$
inside $\Pi_e$) in case $d-e+2\geq 0$; in case $d-e+2<0$ it already
proves $C'_e=C_e$. In case $d-e+2\geq 0$, we find
\begin{displaymath}
\prod_{i=3}^{e-1}\pi_{i,e}.(\pi_{1,e}g_1-\pi_{2,e}g_2)+(c_2-c_1)\varphi_e=0
\end{displaymath}
Since $\pi_{i,e}$ (with $i\in \{3,\ldots,e-1\}$) cannot divide $\varphi_e$, it follows that
$c_1=c_2$ and $\pi_{1,e}g_1=\pi_{2,e}g_2$. So we have that
$g_1=\pi_{2,e}g'_1$ (here we use $\Pi_{1,e}\neq \Pi_{2,e}$), hence
$\varphi'_e=\prod_{i=1}^{e-1}\pi_{i,e}g'_1+c_1\varphi_e$.

Since $\pi_{1,e}$ cannot divide $\varphi'_e$, it follows that
$c_1\neq 0$, thus $$C_e\in \left\langle
\sum_{i=1}^{e-1}\Pi_{i,e}+\mathbb{P}_{\Pi_e,d-e+1}, C'_e\right\rangle\subset
\mathbb{P}_{\Pi_e,d},$$ hence $C_e$ belongs to the linear system on
$\Pi_e$ induced by the linear subsystem of $\mathbb{P}_d$ of
hypersurfaces $X$ satisfying $C_i\subset X$ for $1\leq i\leq e-1$.
From Theorem \ref{theorem main}, it follows that $\mathcal{L}\in
\mathcal{V}_{d,e}$.
\end{proof}

If follows from the considerations in Section \ref{section case III} that in case $$\dim
(\Pi_{i_1}\cap \Pi_{i_2}\cap \cdots \cap \Pi_{i_m})=N-2,$$ the points
$P_{i_1}, P_{i_2}, \ldots ,P_{i_m}$ are collinear. This
situation can also be handled using the arguments from Theorem \ref{theorem main}.

\section{Special case : hypersurfaces with two star points} \label{section two starpoints}

In this Section, we will prove the following result.
\begin{theorem}
Consider the set $\mathcal{S}$ of configurations $\mathcal{L}\in (\mathcal{P}_d)^2$ that are suited for degree $d\geq 3$ in $\mathbb{P}^N$. Then $\mathcal{S}$ has two irreducible components. The set $\mathcal{V}_{d,2}$ is one of these components and it has the expected dimension.
\end{theorem}

First assume $\mathcal{L}=(\Pi_i,P_i,C_i)_{i=1}^2$ is a point in $\mathcal{V}_{d,2}$, so $P_1\not\in \Pi_2$ and $P_2\not\in \Pi_1$. We can introduce coordinates $(X_0:\ldots:X_n)$ on $\PP^N$ such that $\Pi_1$ has equation $X_1=0$, $P_1=(1:0:0:\ldots:0)$, $\Pi_2$ has equation $X_0=0$ and $P_2=(0:1:0:\ldots:0)$. Let $X=Z(f)$ be a hypersurface of degree $d$ in $\PP_d(\mathcal{L})$. We can write $f(X_0,\ldots,X_N)$ as
\begin{multline*} g_{01}(X_0,\ldots,X_N)X_0X_1+g_0(X_0,X_2,\ldots,X_N)X_0\\+g_1(X_1,X_2,\ldots,X_N)X_1+g(X_2,\ldots,X_N). \end{multline*}
Since the equation $g_0(X_0,X_2,\ldots,X_N)X_0+g(X_2,\ldots,X_N)=0$ of the cone $C_1$ in $\Pi_1$ has to be independent of the variable $X_0$, we have $g_0(X_0,X_2,\ldots,X_N)\equiv 0$. Analogously, we get that $g_1(X_1,X_2,\ldots,X_N)\equiv 0$, since $C_2\subset \Pi_2$ is given by $g_1(X_1,X_2,\ldots,X_N)X_1+g(X_2,\ldots,X_N)=0$. So, we conclude that
\begin{equation} \label{equation V_d,2} f=g_{01}(X_0,\ldots,X_N)X_0X_1+g(X_2,\ldots,X_N). \end{equation}
The cones $C_1\subset \Pi_1$ and $C_2\subset \Pi_2$ are both determined by $g(X_2,\ldots,X_N)=0$. Thus the dimension of $\PP_d(\mathcal{L})$ equals ${N+d-2\choose N}$ and $\mathcal{V}_{d,2}$ is irreducible and has dimension $$2N+2(N-1)+{N+d-2\choose N-2}-1,$$ which is the expected dimension.\\

Now assume $\mathcal{L}=(\Pi_i,P_i,C_i)_{i=1}^2\not\in \mathcal{P}_d^{2,0}$, so $P_1\in \Pi_2$ or $P_2\in \Pi_1$. If $P_1\in \Pi_2$, we have that $\langle P_1,P_2\rangle\subset C_2$, hence also $\langle P_1,P_2\rangle\subset C_1$ or $P_2\in \Pi_1$. We may choose coordinates $(X_0:\ldots:X_N)$ on $\PP^N$ such that $P_1=(1:0:0:\ldots:0)$, $P_2=(0:1:0:\ldots:0)$, $\Pi_1:X_2=0$ and $\Pi_2:X_3=0$. Let $X=Z(f)$ be a hypersurface in $\PP_d(\mathcal{L})$. We can write $f$ as \begin{multline*} g_{23}(X_0,\ldots,X_N)X_2X_3+g_2(X_0,X_1,X_2,X_4,\ldots,X_N)X_2\\+g_3(X_0,X_1,X_3,X_4,\ldots,X_N)X_3+g(X_0,X_1,X_4,\ldots,X_N).\end{multline*}
Since $C_1\subset \Pi_1$ is given by $$g_3(X_0,X_1,X_3,X_4,\ldots,X_N)X_3+g(X_0,X_1,X_4,\ldots,X_N)=0,$$ we have that $g_3$ and $g$ are independent of $X_0$. Analogously, by considering $C_2\subset \Pi_2$, we get that $g_2$ and $g$ are independent of $X_1$. We conclude that $f$ can be written as
\begin{multline} \label{equation line in X} g_{23}(X_0,\ldots,X_N)X_2X_3+g_2(X_0,X_2,X_4,\ldots,X_N)X_2\\+g_3(X_1,X_3,X_4,\ldots,X_N)X_3+g(X_4,\ldots,X_N) \end{multline}

If $N=3$, we have $g\equiv 0$ and $$f=g_{23}(X_0,X_1,X_2,X_3)X_2X_3+g_2(X_0,X_2)X_2+g_3(X_1,X_3)X_3.$$
The cones $C_1\subset\Pi_1$ and $C_2\subset \Pi_2$ are determined by respectively $g_3=0$ and $g_2=0$, so the dimension of $\PP_d(\mathcal{L})$ equals ${d+1\choose 3}+1$ and the dimension of the irreducible locus in $\mathcal{P}_d^2$ is $2.3+2.1+2(d-1)=2(d+3)$.

If $N>3$, the dimension of $\PP_d(\mathcal{L})$ is ${N+d-2\choose N}$ and the irreducible locus in $\mathcal{P}_d^2$ has dimension equal to $$2N+2(N-2)+{N+d-4\choose N-4}+2{N+d-3\choose N-2}-1.$$

\section{Special case : hypersurfaces with three star points} \label{section three starpoints}

\subsection{Components of $\mathcal{V}_{d,3}$}

In this section (including Subsections \ref{section case I}-\ref{section case IV}), we will prove the following theorem.
\begin{theorem}
The configuration space $\mathcal{V}_{d,3}$ (where $d\geq 3$) has $2d-2$ irreducible components, of which $\phi(d)+\phi(d-1)$ for $N=3$ and $\phi(d)$ for $N>3$ are of the expected dimension.
\end{theorem}

Let $\mathcal{L}=(\Pi_i,P_i,C_i)_{i=1}^3$ be an element of $\mathcal{V}_{d,3}$, so $P_i\not\in \Pi_j$ for $i\neq j$. We can choose coordinates $(X_0:\ldots:X_N)$ on $\PP^N$ such that $P_1=(1:0:0:\ldots:0)$, $\Pi_1$ has equation $X_1=0$, $P_2=(0:1:0:\ldots:0)$ and $\Pi_2$ has equation $X_0=0$. Let $X=Z(f)$ be a hypersurface in $\PP_d(\mathcal{L})$. Using Section \ref{section two starpoints}, we have that $f$ is of the form $g_{01}(X_0,\ldots,X_N)X_0X_1+g(X_2,\ldots,X_N)$.

\subsubsection{Case I: $P_3\not\in\langle P_1,P_2\rangle$ and $\Pi_1\cap\Pi_2\not\subset\Pi_3$} \label{section case I}

In this case, we may assume that $P_3=(a:b:a+b:0:\ldots:0)$ and $\Pi_3$ has equation $X_2=X_0+X_1$. Note that $a\neq 0$, $b\neq 0$ and $a+b\neq 0$ since respectively $P_3\not \in \Pi_2$, $P_3\not\in \Pi_1$ and $P_3\not\in\langle P_1,P_2\rangle$. So we may take $a=-1$ and $b=t$ with $t\neq 1$ and $t\neq 0$.

Assume for simplicity that $N=3$ and consider the coordinate transformation defined by
\begin{equation*} \left\{ \begin{array}{lll} Y_0 &=& tX_0+X_1 \\ Y_1 &=& X_0+X_1-X_2 \\ Y_2 &=& X_2 \\ Y_3 &=& X_3\end{array}\right. \quad \Longleftrightarrow \quad
\left\{ \begin{array}{lll} X_0 &=& \frac{1}{t-1}(Y_0-Y_1-Y_2) \\ X_1 &=& \frac{1}{t-1}(-Y_0+tY_1+tY_2) \\ X_2 &=& Y_2 \\ X_3 &=& Y_3\end{array}\right.,
\end{equation*}
so for the new coordinate system, we have $P_3=(0:0:1:0)$ and $\Pi_3:Y_1=0$. The equation of $X$ becomes
\begin{multline}
\frac{(Y_0-Y_1-Y_2)(-Y_0+tY_1+tY_2)}{(t-1)^2}g_{01}\left(\frac{Y_0-Y_1-Y_2}{t-1},\frac{-Y_0+tY_1+tY_2}{t-1},Y_2,Y_3\right)\\+g(Y_2,Y_3)=0,
\end{multline}
so the cone $C_3\subset \Pi_3$ is given by
\begin{equation} \label{equation C_3 case I} \frac{(Y_0-Y_2)(-Y_0+tY_2)}{(t-1)^2}g_{01}\left(\frac{Y_0-Y_2}{t-1},\frac{-Y_0+tY_2}{t-1},Y_2,Y_3\right)\\+g(Y_2,Y_3)=0.
\end{equation}
The left hand side of \eqref{equation C_3 case I} should be independent of the variable $Y_2$. If we write $$g_{01}\left(\frac{Y_0-Y_2}{t-1},\frac{-Y_0+tY_2}{t-1},Y_2,Y_3\right)=\sum_{\substack{i,j\geq 0 \\ i+j\leq d-2}}\, c_{i,j} Y_0^i Y_2^j Y_3^{d-i-j-2}$$ and
$$g(Y_2,Y_3) = \sum_{0\leq j\leq d}\,d_j Y_2^j Y_3^{d-j},$$
the equation \eqref{equation C_3 case I} is independent of $Y_2$ if and only if
\begin{equation} \label{system eq coef}
\left\{ \begin{array}{ll} -tc_{i,j-2} + (t+1)c_{i-1,j-1} - c_{i-2,j} = 0 &\text{for } i,j\geq 2 \\
(t+1)c_{i-1,0} - c_{i-2,1}= 0 &\text{for } i\geq 2 \\
-tc_{1,j-2} + (t+1)c_{0,j-1} = 0 &\text{for } j\geq 2 \\
\frac{-t}{(t-1)^2} c_{0,j-2}+ d_j = 0 &\text{for } j\geq 2 \\
(t-1)c_{0,0} = d_1 = 0
\end{array} \right.
\end{equation}
So for each $j\in\{2,\ldots,d\}$, the coefficients $c_{j-2,0},c_{j-3,1},\ldots,c_{0,j-2}$ satisfy a system $\Sigma_j$ of linear equations, for which the coefficient matrix is equal to the following tridiagonal matrix
\begin{equation*}
M_j = \begin{bmatrix} t+1 & -1 & \ddots & 0 & 0 \\ -t & t+1 & \ddots & 0 & 0 \\ \ddots & \ddots & \ddots & \ddots & \ddots \\ 0 & 0 & \ddots & t+1 & -1\\ 0 & 0 & \ddots & -t & t+1 \end{bmatrix} \in \mathbb{C}^{(j-1)\times (j-1)}.
\end{equation*}
The determinant of $M_j$ equals $$\frac{t^j-1}{t-1} = 1+t+t^2+\ldots+t^{j-1},$$
so $\Sigma_j$ has a non-zero solution if and only if $t^j=1$. In this case, the solutions of $\Sigma_j$ are given by
\begin{equation*}\begin{array}{lll} (c_{j-2,0},c_{j-3,1},\ldots,c_{0,j-2})&=&A_j.(1,1+t,\ldots,1+t+t^2+\ldots+t^{j-1}) \\ &=&A_j(\frac{t-1}{t-1},\frac{t^2-1}{t-1},\ldots,\frac{t^{j-1}-1}{t-1}),\end{array} \end{equation*}
where $A_j\in \mathbb{C}$, hence
\begin{equation*}\begin{array}{lll}
\sum_{k=0}^{j-2}\, c_{j-k-2,k} Y_0^{j-k-2} Y_2^k &=& A_j.\left(\sum_{k=0}^{j-2}\,\frac{t^{k+1}-1}{t-1} Y_0^{j-k-2} Y_2^k \right) \\
&=& \frac{A_j}{t-1} \left(t.\frac{Y_0^{j-1}-(tY_2)^{j-1}}{Y_0-tY_2} - \frac{Y_0^{j-1}-Y_2^{j-1}}{Y_0-Y_2} \right) \\
&=& A_j.\frac{Y_0^j-Y_2^j}{(Y_0-Y_2)(Y_0-tY_2)}. \end{array}
\end{equation*}
So we get $$g_{01}\left(\frac{Y_0-Y_2}{t-1},\frac{-Y_0+tY_2}{t-1},Y_2,Y_3\right) = \sum_{\substack{0<j\leq d\\t^j=1}}\,A_j\frac{(Y_0^j-Y_2^j)Y_3^{d-j}}{(Y_0-Y_2)(Y_0-tY_2)}.$$
This implies that $g_{01}(X_0,X_1,X_2,X_3)$ is of the form
$$\sum_{\substack{0<j\leq d\\t^j=1}}\,A_j\frac{((tX_0+X_1)^j-X_2^j)X_3^{d-j}}{(tX_0+X_1-X_2)(tX_0+X_1-tX_2)} + (X_0+X_1-X_2)g_{012},$$
where $g_{012}\in\mathbb{C}[X_0,X_1,X_2,X_3]_{d-3}$.

Since $$c_{0,j-2}=A_j.\frac{t^{j-1}-1}{t-1}=\frac{-A_j}{t},$$ from \eqref{system eq coef} follows that $d_j=\frac{-A_j}{(t-1)^2}$, thus
$$g(X_2,X_3)=\frac{-1}{(t-1)^2} \sum_{\substack{0\leq j\leq d\\ t^j=1}} A_j X_2^j X_3^{d-j},$$
where we take $A_0=-(t-1)^2 d_0$.

If $t^{d}\neq 1$ and $t^{d-1}\neq 1$, we see that $f$ belongs to the ideal $I:=<X_0X_1,X_3^2>$, hence $P(0:0:1:0)\in \Sing(X)$, a contradiction. So we have either $t^d=1$ or $t^{d-1}=1$. In these cases, for general $A_j\in\mathbb{C}$ and $g_{012}\in\mathbb{C}[X_0,X_1,X_2,X_3]_{d-3}$, the surface $X=Z(f)$ will be smooth.

If $t^d=1$ or $t^{d-1}=1$, denote the component of $\mathcal{V}_{d,3}$ corresponding to the root of unity $t\neq 1$ by $\mathcal{V}_t$ and denote the order of $t$ in $\mathbb{C}^{\star}$ by $\theta(t)$. If we fix $\mathcal{L}\in \mathcal{V}_t$, the values of the numbers $A_j$ are fixed (up to a scalar), but $g_{012}$ can vary, so the dimension of $\PP_d(\mathcal{L})$ equals ${d\choose 3}$. In case $t^d=1$, the dimension of $\mathcal{V}_t$ equals $$3.3+2.2+1+\left(\frac{d}{\theta(t)}+1\right)-1=14+\frac{d}{\theta(t)},$$ since we can choose arbitrarily $\Pi_1,\Pi_2,\Pi_3\subset \PP^3$, $P_1\in\Pi_1$, $P_2\in\Pi_2$, $P_3$ on a certain line in $\Pi_3$ and $A_0,A_{\theta(t)},\ldots,A_d$ (up to a scalar). Analogously, we can see that in case $t^{d-1}=1$, the dimension of $\mathcal{V}_t$ equals $14+\frac{d-1}{\theta(t)}$. Note that in both cases, the dimension of $\mathcal{V}_t$ is dependent of $\theta(t)$ and it is equal to the expected dimension $15$ if and only if $\theta(t)$ is maximal (i.e. $\theta(t)\in\{d-1,d\}$).\\

These results can easily be generalized to $N>3$. Indeed, the defining polynomial $f$ of $X$ becomes
\begin{multline*}
X_0 X_1.\sum_{0<j\leq d;\ t^j=1}\, A_j(X_3,\ldots,X_N)\frac{(tX_0+X_1)^j-X_2^j}{(tX_0+X_1-X_2)(tX_0+X_1-tX_2)} + \\
X_0 X_1 (X_0+X_1-X_2) g_{012}(X_0,\ldots,X_N) -  \sum_{0\leq j\leq d;\ t^j=1}\, \frac{A_j(X_3,\ldots,X_N) X_2^j}{(t-1)^2},
\end{multline*}
where $A_j$ is a polynomial of degree $d-j$.

If $t^d\neq 1$ and $t^{d-1}\neq 1$, we see that $f$ belongs to the ideal $I$ generated by $X_0X_1$ and $X_iX_j$ with $3\leq i\leq j\leq N$, thus $P(0:0:1:0:\ldots:0)\in\Sing(X)$, a contradiction. So we have either $t^d=1$ or $t^{d-1}=1$. In this case, let $\theta(t)$ be the order of $t$ in $\mathbb{C}^{\star}$ and write $\mathcal{V}_t$ to denote the component corresponding to $t\neq 1$. The dimension of $\PP_d(\mathcal{L})$ is ${N+d-3\choose N}$ and $\mathcal{V}_t$ has dimension equal to
\begin{multline*} 3N+2(N-1)+(N-2)+\sum_{0\leq j\leq d;\ t^j=1} {N+d-3-j\choose N-3} - 1 \\ = 6N+\sum_{0\leq j\leq d;\ t^j=1} {N+d-3-j\choose N-3}-5. \end{multline*}
The component $\mathcal{V}_t$ has the expected dimension, which is equal to $$6N+{N+d-3\choose N-3}-4,$$ if and only if $t$ is a primitive $d$th root of unity (i.e. $\theta(t)=d$).

\subsubsection{Case II: $P_3\in\langle P_1,P_2\rangle$ and $\Pi_1\cap\Pi_2\not\subset\Pi_3$} \label{section case II}

We may assume that $P_3=(1:-1:0:\ldots:0)$ and $\Pi_3:X_2=X_0+X_1$. Consider the coordinate transformation defined by
\begin{equation*} \left\{ \begin{array}{lll} Y_0 &=& X_0 \\ Y_1 &=& X_0+X_1 \\ Y_2 &=& X_2-X_0-X_1 \\ Y_i &=& X_i \quad (\forall i\geq 3)\end{array}\right. \quad \Longleftrightarrow \quad
\left\{ \begin{array}{lll} X_0 &=& Y_0 \\ X_1 &=& Y_1-Y_0 \\ X_2 &=& Y_1+Y_2 \\ X_i &=& Y_i \quad (\forall i\geq 3) \end{array}\right.,
\end{equation*}
hence in the new system, we have $P_3=(1:0:0:\ldots:0)$, $\Pi_3:Y_2=0$ and $X$ is defined by $$Y_0(Y_1-Y_0)g_{01}(Y_0,Y_1-Y_0,Y_1+Y_2,Y_3,\ldots,Y_N)+ g(Y_1+Y_2,Y_3,\ldots,Y_N)=0.$$ The cone $C_3\subset \Pi_3$ is given by
\begin{equation} \label{equation C_3 case II} Y_0(Y_1-Y_0)g_{01}(Y_0,Y_1-Y_0,Y_1,Y_3,\ldots,Y_N)+ g(Y_1,Y_3,\ldots,Y_N)=0.\end{equation}
Since \eqref{equation C_3 case II} has to be independent of the variable $Y_0$, we get that $g_{01}(Y_0,Y_1-Y_0,Y_1,Y_3,\ldots,Y_N)\equiv 0$, thus $g_{01}(X_0,\ldots,X_N)$ is divisible by $X_2-X_0-X_1$ and $f$ is of the form $$X_0X_1(X_2-X_0-X_1)g_{012}(X_0,\ldots,X_N)+g(X_2,\ldots,X_N).$$

This case gives rise to a component of $\mathcal{V}_{d,3}$ not described in Case I, which we will denote by $\mathcal{V}_1$.
If we fix $\mathcal{L}$, the polynomial $g$ is fixed, but $g_{012}$ can vary. So the dimension of $\PP_d(\mathcal{L})$ is equal to ${N+d-3\choose N}$ and the dimension of $\mathcal{V}_1$ is $$3N+2(N-1)+{N+d-2\choose N-2}-1.$$

\subsubsection{Case III: $P_3\not\in\langle P_1,P_2\rangle$ and $\Pi_1\cap\Pi_2\subset\Pi_3$} \label{section case III}

We can take the coordinates on $\PP^N$ so that $P_3=(0:0:1:0:\ldots:0)$ and $\Pi_3:X_0=X_1$. The cone $C_3$ is defined by $X_1-X_0=0$ and
\begin{equation} \label{equation C_3 case III} X_0^2g_{01}(X_0,X_0,X_2,X_3,\ldots,X_N)+g(X_2,\ldots,X_N)=0.\end{equation}
Since \eqref{equation C_3 case III} has to be independent of $X_2$, we get that $g$ is independent of $X_2$ and $P_2\in\text{Sing}(X)$, a contradiction. So there are no elements $\mathcal{L}\in \mathcal{V}_{d,3}$ of this form.

\subsubsection{Case IV: $P_3\in\langle P_1,P_2\rangle$ and $\Pi_1\cap\Pi_2\subset\Pi_3$} \label{section case IV}

Assume that $P_3=(1:1:0:\ldots:0)$ and $\Pi_3: X_0=X_1$. Consider the following coordinate transformation
\begin{equation*} \left\{ \begin{array}{lll} Y_0 &=& X_0 \\ Y_1 &=& X_1-X_0 \\ Y_i &=& X_i \quad (\forall i\geq 2)\end{array}\right. \quad \Longleftrightarrow \quad
\left\{ \begin{array}{lll} X_0 &=& Y_0 \\ X_1 &=& Y_0+Y_1 \\ X_i &=& Y_i \quad (\forall i\geq 2) \end{array}\right..
\end{equation*}
For the new coordinates, we have $P_3=(1:0:\ldots:0)$, $\Pi_3:Y_1=0$ and $X$ is given by
$$Y_0(Y_0+Y_1)g_{01}(Y_0,Y_0+Y_1,Y_2,\ldots,Y_N)+g(Y_2,\ldots,Y_N)=0.$$ The equation of the cone $C_3\subset \Pi_3$ is
\begin{equation} \label{equation C_3 case IV} Y_0^2g_{01}(Y_0,Y_0,Y_2,\ldots,Y_N)+g(Y_2,\ldots,Y_N)=0.\end{equation}
Since \eqref{equation C_3 case IV} has to be independent of $Y_0$, we get $g_{01}(Y_0,Y_0,Y_2,\ldots,Y_N)\equiv 0$ and so $g_{01}(X_0,\ldots,X_N)$ is divisible by $X_1-X_0$. Thus, $f$ is of the form
$$X_0X_1(X_1-X_0)g_{01}(X_0,\ldots,X_N)+g(X_2,\ldots,X_N).$$

It is clear that in this case, the elements $\mathcal{L}$ are contained in the component $\mathcal{V}_1$. Moreover, this subset of $\mathcal{V}_1$ has dimension $$2N+2(N-1)+1+{N+d-2\choose N-2}-1.$$

\subsubsection{Conclusion} \label{section conclusion}

The set $\mathcal{V}_{d,3}$ has $(d-1)+(d-2)+1=2d-2$ components, namely $\mathcal{V}_t$ with $t$ a root of the polynomial $$\frac{(t^d-1)(t^{d-1}-1)}{t-1}\in \mathbb{C}[t].$$ The component $\mathcal{V}_t$ has the expected dimension if and only if $\theta(t)\in\{d-1,d\}$ if $N=3$ and $\theta(t)=d$ if $N>3$. So, $\mathcal{V}_{d,3}$ has $\phi(d)+\phi(d-1)$ components of the expected dimension if $N=3$ and $\phi(d)$ if $N>3$.

\begin{remark}
For smooth cubic surfaces (i.e. $d=N=3$), the set $\mathcal{V}_{d,3}$ has three components of the expected dimension $15$ (namely $\mathcal{V}_{-1}$, $\mathcal{V}_{\omega_3}$ and $\mathcal{V}_{\omega_3^2}$, where $\omega_3=e^{2\pi i/3}$) and one component $\mathcal{V}_1\subset \mathcal{V}_{d,3}$ of dimension $16$.
\end{remark}

\begin{remark}
Consider the Fermat hypersurface $X_{d,N}$ (see Example \ref{exa2}). Let $i,j,k$ be different elements in $\{0,\ldots,N\}$ and $\xi_1,\xi_2,\xi_3\in\mathbb{C}$ such that $\xi_1^d=\xi_2^d=\xi_3^d=-1$. Then the triple of star points $$(E_{i,j}(\xi_1),E_{i,j}(\xi_2),E_{i,k}(\xi_3)),$$ with $\xi_1\neq \xi_2$, belongs to the component $\mathcal{V}_t$ with $t=\frac{\xi_1}{\xi_2}$. On the other side, the triple of star points
$$(E_{i,j}(\xi_1),E_{i,k}(\xi_2),E_{j,k}(\xi_3))$$ belongs to $\mathcal{V}_t$ with $t=(\frac{\xi_2}{\xi_1\xi_3})^2$. \end{remark}

\subsection{Components outside $\mathcal{V}_{d,3}$ : Intermediate case}

We show the following result for the intermediate case.
\begin{theorem} \label{thm extremal}
Let $\mathcal{V}_{\text{int}}$ be the set of configurations $\mathcal{L}\in (\mathcal{P}_d)^3$ that are suited for degree $d\geq 3$ in $\mathbb{P}^N$ and such that $P_1\not\in\Pi_2\cup\Pi_3$ but $\langle P_2,P_3\rangle\subset \Pi_2\cap\Pi_3$. Then $\mathcal{V}_{\text{int}}$ is irreducible of dimension $$3N+(N-1)+2(N-2)+{N+d-3\choose N-3}+\sum_{k=1}^{d-1}\,{N+d-k-4\choose N-3}-1.$$
\end{theorem}

Let $\mathcal{L}\in (\mathcal{P}_d)^3$ with $\PP_d(\mathcal{L})\neq\emptyset$ such that $P_1\not\in\Pi_2\cup\Pi_3$ and $P_2,P_3\not\in \Pi_1$, but $P_2\in \Pi_3$. In this case, we have that $\langle P_2,P_3\rangle\subset \Pi_2\cap\Pi_3$.

We can choose coordinates $(X_0:\ldots:X_N)$ such that $P_1=(1:0:0:\ldots:0)$, $\Pi_1$ has equation $X_1=0$, $P_2=(0:1:0:\ldots:0)$, $\Pi_2$ has equation $X_0=0$, $P_3=(0:1:1:0:\ldots:0)$ and $\Pi_3$ has equation $X_0=X_3$. From Section \ref{section two starpoints} follows that $f$ is of the form $$X_0X_1g_{01}(X_0,\ldots,X_N)+g(X_2,\ldots,X_N).$$

Consider the following coordinate transformation \begin{equation*} \left\{ \begin{array}{lll} Y_0 &=& X_0-X_3 \\ Y_1 &=& X_1-X_2 \\ Y_i &=& X_i \quad (\forall i\geq 2)\end{array}\right. \quad \Longleftrightarrow \quad
\left\{ \begin{array}{lll} X_0 &=& Y_0+Y_3 \\ X_1 &=& Y_1+Y_2 \\ X_i &=& Y_i \quad (\forall i\geq 2) \end{array}\right..
\end{equation*}
In this system, $P_3=(0:0:1:0:\ldots:0)$, $\Pi_3:Y_0=0$ and $X$ has equation $$(Y_0+Y_3)(Y_1+Y_2)g_{01}(Y_0+Y_3,Y_1+Y_2,Y_2,Y_3,\ldots,Y_N)+g(Y_2,Y_3,\ldots,Y_N)=0,$$ so the cone $C_3\subset \Pi_3$ is given by \begin{equation} \label{equation C_3 intermediate} Y_3(Y_1+Y_2)g_{01}(Y_3,Y_1+Y_2,Y_2,Y_3,\ldots,Y_N)+g(Y_2,Y_3,\ldots,Y_N)=0. \end{equation}

For simplicity, we will first assume $N=3$. If we write $$g_{01}(Y_3,Y_1+Y_2,Y_2,Y_3)=\sum_{\substack{i,j\geq 0 \\ i+j\leq d-2}}\, a_{i,j} Y_1^i Y_2^j Y_3^{d-i-j-2}$$ and $g(Y_2,Y_3)=\sum_{j=0}^d\,b_j Y_2^j Y_3^{d-j}$, we see that the equation \eqref{equation C_3 intermediate} of $C_3$ is independent of the variable $Y_2$ if and only if
\begin{equation} \label{system eq coef 2} \left\{ \begin{array}{ll} a_{i-1,j}+a_{i,j-1}=0 & (\forall i>0,j>0 \text{ with } i+j\leq d-1) \\
b_j+a_{0,j-1}=0 & (\forall j\in \{1,\ldots,d-1\}) \\ b_d=0 \end{array}\right.. \end{equation} If \eqref{system eq coef 2} is satisfied, we have $a_{i,k-i-1}=-(-1)^ib_{k}$ for all $1\leq k\leq d-1$, and thus $$\sum_{i=0}^{k-1}\,a_{i,k-i-1} Y_1^i Y_2^{k-i-1}=-b_k \sum_{i=0}^k\,(-Y_1)^i Y_2^{k-i-1}= -b_k \frac{Y_2^k-(-Y_1)^k}{Y_2+Y_1}.$$ It follows $$g_{01}(Y_3,Y_1+Y_2,Y_2,Y_3) = -\frac{1}{Y_2+Y_1} \sum_{k=1}^{d-1}\,b_k(Y_2^k-(-Y_1)^k)Y_3^{d-k-1},$$ so $g_{01}(X_0,X_1,X_2,X_3)$ is of the form $$(X_3-X_0)g_{013}(X_0,\ldots,X_N)-\frac{1}{X_1} \sum_{k=1}^{d-1}\,b_k(X_2^k-(X_2-X_1)^k)X_3^{d-k-1},$$ where $g_{013}$ is a polynomial of degree $d-3$. We conclude that $f$ is of the form
$$ X_0X_1(X_3-X_0)g_{013}-X_0 \sum_{k=1}^{d-1}\,b_k(X_2^k-(X_2-X_1)^k)X_3^{d-k-1} + \sum_{k=0}^{d-1}\,b_k X_2^k X_3^{d-k}.$$ So, in this case, for $N=3$, the dimension of $\PP_d(\mathcal{L})$ is ${d\choose 3}$ and we get a component of the set of surfaces with three star points of dimension $$3.3+2+2.1+d+{d\choose 3}-1=12+d+{d\choose 3}.$$

These results can easily be generalized to general values of $N\geq 3$. Indeed, in this case, $f$ is of the form $$X_0X_1(X_3-X_0)g_{013}-X_0 \sum_{k=1}^{d-1}\,B_k(X_2^k-(X_2-X_1)^k) + X_3 \sum_{k=1}^{d-1}\,B_k X_2^k + B_0,$$ where the polynomials $B_i$ are homogeneous in the variables $X_3,\ldots,X_N$ of degree $d$ for $i=0$ and $d-k-1$ for $k\in\{1,\ldots,d-1\}$. We get a component of the set of hypersurfaces with three star points of dimension $$3N+(N-1)+2(N-2)+{N+d-3\choose N}+{N+d-3\choose N-3}+\sum_{k=1}^{d-1}\,{N+d-k-4\choose N-3}-1.$$

\subsection{Components outside $\mathcal{V}_{d,3}$ : Extremal case}

We show the following result for the extremal case.
\begin{theorem}
Let $\mathcal{V}_{\text{ext}}$ be the set of configurations $\mathcal{L}\in (\mathcal{P}_d)^3$ that are suited for degree $d$ in $\mathbb{P}^N$ and such that $P_i\in\Pi_j$ for all $i,j\in\{1,2,3\}$. If $d\geq 6$, then $\mathcal{V}_{\text{ext}}$ has two irreducible components. The first component $\mathcal{V}_{\text{ext,I}}$ corresponds to configurations $\mathcal{L}$ with $\Pi_1,\Pi_2,\Pi_3$ linearly independent and has dimension $$3N+3(N-3)+3{N+d-4\choose N-2}+3{N+d-5\choose N-4}+{N+d-6\choose N-6}-1;$$ the second component $\mathcal{V}_{\text{ext,II}}$ corresponds to configurations $\mathcal{L}$ with $\Pi_1,\Pi_2,\Pi_3$ linearly dependent and has dimension $$3N+2(N-3)+{N+d-3\choose N-1}+2{N+d-4\choose N-3}+{N+d-5\choose N-3}+{N+d-5\choose N-5}.$$
\end{theorem}

Assume $\mathcal{L}\in (\mathcal{P}_d)^3$ with $\PP_d(\mathcal{L})\neq\emptyset$ and $P_i\in\Pi_j$ for all $i,j\in\{1,2,3\}$. In this case, it is easy to see that $$L=\langle P_1,P_2,P_3\rangle \subset \Pi_1\cap \Pi_2 \cap \Pi_3\cap X = C_1\cap C_2\cap C_3.$$
Note that $P_3\not\in \langle P_1,P_2\rangle$ (see Proposition \ref{prop2}) and that $N>4$ (for $N=4$, the plane $L$ would be a component of the cones $C_i$, which contradicts Lemma \ref{lem1}).

Choose coordinates $(X_0:\ldots:X_N)$ on $\PP^N$ so that $P_1=(1:0:0:\ldots:0)$, $P_2=(0:1:0\ldots:0)$, $P_3=(0:0:1:\ldots:0)$, $\Pi_1$ has equation $X_3=0$ and $\Pi_2$ has equation $X_4=0$.

\subsubsection{Case I: $\Pi_1\cap\Pi_2 \not\subset \Pi_3$} \label{subsection extr I}

We can assume that $\Pi_3:X_5=0$. We can write $$f=X_3X_4X_5g_{345}+X_3X_4g_{34}+X_3X_5g_{35}+X_4X_5g_{45}+X_3g_3+X_4g_4+X_5g_5+g,$$
where for example $g_{34}$ is independent of $X_5$, $g_3$ is independent of $X_4$ and $X_5$ and $g$ is independent of $X_3$, $X_4$ and $X_5$.

Since $C_1\subset \Pi_1$ is defined by $X_4X_5g_{45}+X_4g_4+X_5g_5+g=0$, we get that $g_{45},g_4,g_5,g$ are independent of the variable $X_0$.
Analogously, we get that $g_{35},g_3,g_5,g$ are independent of $X_1$ and $g_{34},g_3,g_4,g$ independent of $X_2$, so we conclude $f$ is of the form
\begin{equation*} \begin{array}{c}
X_3X_4X_5g_{345}(X_0,\ldots,X_N)+ X_3X_4g_{34}(X_0,X_1,X_3,X_4,X_6,\ldots,X_N)+ \\ X_3X_5g_{35}(X_0,X_2,X_3,X_5,\ldots,X_N)+ X_4X_5g_{45}(X_1,X_2,X_4,X_5,\ldots,X_N)+ \\ X_3g_3(X_0,X_3,X_6,\ldots,X_N)+X_4g_4(X_1,X_4,X_6,\ldots,X_N)+ \\X_5g_5(X_2,X_5,X_6,\ldots,X_N)+g(X_6,\ldots,X_N).\end{array}\end{equation*}

When we fix the element $\mathcal{L}$, the polynomials $g_{34},g_{35},g_{45},g_3,g_4,g_5,g$ are fixed, but $g_{345}$ can still vary, so the dimension of $\PP_d(\mathcal{L})$ is equal to ${N+d-3\choose N}$. This case gives rise to a component of the set of hypersurfaces with three star points of dimension $d_{\text{ext,I}}$ equal to $$3N+3(N-3)+{N+d-3\choose N}+3{N+d-4\choose N-2}+3{N+d-5\choose N-4}+{N+d-6\choose N-6}-1.$$

\subsubsection{Case II: $\Pi_1\cap\Pi_2 \subset \Pi_3$}

We may assume $\Pi_3$ is the hyperplane $X_4-X_3=0$. Write $f$ as $X_3X_4g_{34}+X_3g_3+X_4g_4+g$, with $g_3$ independent of $X_4$, $g_4$ of $X_3$ and $g$ of $X_3$ and $X_4$.

Since $C_1\subset \Pi_1$ is given by $X_4g_4+g$, we see that $g_4$ and $g$ are independent of $X_0$. By considering $C_2\subset \Pi_2$, we get that $g_3$ and $g$ are independent of $X_1$. The cone $C_3$ is has equations $X_4-X_3=f=0$, so \begin{multline*} X_3^2g_{34}(X_0,X_1,X_2,X_3,\underline{X_3},X_5,\ldots,X_N)+X_3g_3(X_0,X_2,X_3,X_5,\ldots,X_N)+\\ X_3g_4(X_1,X_2,\underline{X_3},X_5,\ldots,X_N)+g(X_2,X_5,\ldots,X_N)
\end{multline*} is independent of the variable $X_2$. This implies that $g$ is independent of $X_2$. If we write $g_i=X_ih_i+h'_i$ for $i\in\{3,4\}$ with $h'_i$ independent of $X_i$, we see that $$h'_3(X_0,X_2,X_5,\ldots,X_N)+h'_4(X_1,X_2,X_5,\ldots,X_N)$$ and \begin{multline*} g_{34}(X_0,X_1,X_2,X_3,\underline{X_3},X_5,\ldots,X_N)+h_3(X_0,X_2,X_3,X_5,\ldots,X_N)\\+h_4(X_1,X_2,\underline{X_3},X_5,\ldots,X_N) \end{multline*} are independent of $X_2$. If we write $h'_3=X_2b_2+b$ and $h'_4=X_2c_2+c$ with $b$ and $c$ independent of $X_2$, we have that $b_2+c_2\equiv 0$. Hence, $b_2$ is independent of $X_0$ and $c_2$ is independent of $X_1$. On the other hand, if $g_{34}=X_2a_2+a$, $h_3=X_2b_{23}+b_3$ and $h_4=X_2c_{24}+c_4$ with $a$, $b_3$ and $c_4$ independent of $X_2$, we get \begin{multline*} a_2(X_0,X_1,X_2,X_3,\underline{X_3},X_5,\ldots,X_N)+b_{23}(X_0,X_2,X_3,X_5,\ldots,X_N)\\+c_{24}(X_1,X_2,\underline{X_3},X_5,\ldots,X_N)\equiv 0. \end{multline*}
We see that $a_2(X_0,X_1,X_2,X_3,\underline{X_3},X_5,\ldots,X_N)$ does not contain terms divisible by $X_0X_1$, hence we can write $a_2$ as $X_0X_1(X_4-X_3)a_{012}+X_0a_{02}+X_1a_{12}+a'$, with $a_{02}$ independent of $X_1$, $a_{12}$ independent of $X_0$ and $a'_2$ independent of $X_0$ and $X_1$. Also write $b_{23}$ as $X_0b_{023}+b'_{23}$ with $b'_{23}$ independent of $X_0$ and $c_{24}$ as $X_1c_{124}+c'_{24}$ with $c'_{24}$ independent of $X_1$. We have that $$b_{023}(X_0,X_2,X_3,X_5,\ldots,X_N)\equiv -a_{02}(X_0,X_2,X_3,\underline{X_3},X_5,\ldots,X_N),$$ $$c_{124}(X_1,X_2,\underline{X_3},X_5,\ldots,X_N)\equiv -a_{12}(X_1,X_2,X_3,\underline{X_3},X_5,\ldots,X_N)$$ or $$c_{124}(X_1,X_2,X_4,X_5,\ldots,X_N)\equiv -a_{12}(X_1,X_2,\underline{X_4},X_4,X_5,\ldots,X_N)$$ and \begin{multline*} b'_{23}(X_2,X_3,X_5,\ldots,X_N)\equiv \\-a'_2(X_2,X_3,\underline{X_3},X_5,\ldots,X_N)-c'_{24}(X_2,\underline{X_3},X_5,\ldots,X_N).\end{multline*}
We conclude that $f$ is of the form \begin{multline*} X_3X_4(X_2a_2+a)+X_3(X_2X_3b_{23}+X_3b_3+X_2b_2+b)+\\ X_4(X_2X_4(X_1c_{124}+c'_{24})+X_4c_4+X_2c_2+c)+g,\end{multline*} with the above conditions (in particular, $b_{23}$, $b_2$ and $c_{124}$ are fixed, given $a_2$, $c_2$ and $c'_{24}$).

In this case, the dimension of $\PP_d(\mathcal{L})$ is again ${N+d-3\choose N}$. The set of hypersurfaces, corresponding to star point configurations $\mathcal{L}\in(\mathcal{P}_d)^3$ with $P_i\in \Pi_j$ and $\Pi_1,\Pi_2,\Pi_3$ linearly dependent, gives rise to an irreducible locus of dimension $d_{\text{ext,II}}$ equal to $$3N+2(N-3)+{N+d-2\choose N}+2{N+d-4\choose N-3}+{N+d-5\choose N-3}+{N+d-5\choose N-5}.$$
Since $$d_{\text{ext,II}}=d_{\text{ext,I}}+4-N+{N+d-6\choose N-1},$$ we have that $d_{\text{ext,II}}>d_{\text{ext,I}}$ for $d\geq 6$. This implies Theorem \ref{thm extremal}. The authors expect that the statement of the theorem also holds for $d\in\{3,4,5\}$, although is this case $d_{\text{ext,II}}<d_{\text{ext,I}}$ (except for the case $N=d=5$, where we have $d_{\text{ext,II}}=d_{\text{ext,I}}=116$).

\section*{Acknowledgements}
The authors would like to thank L. Chiantini, since the generalization of total inflection points on plane curves to higher dimensional varieties was motivated by a question of him. Both authors are partially supported by the Fund of Scientific Research - Flanders
(G.0318.06).

\end{document}